\documentclass[reqno]{amsart}
\usepackage{txfonts}
\usepackage{CJK}
\usepackage{hyperref}
\usepackage{cite}
\usepackage{amsmath}
\usepackage{mathrsfs}

\numberwithin{equation}{section}

\newtheorem{theorem}{Theorem}[section]
\newtheorem{lemma}[theorem]{Lemma}
\newtheorem{definition}[theorem]{Definition}

\newtheorem{proposition}[theorem]{Proposition}
\newtheorem{remark}[theorem]{Remark}
\newtheorem{corollary}[theorem]{Corollary}

\allowdisplaybreaks

\begin{document}
	
\title[\hfil On weak and viscosity solutions of nonlocal double phase equations] {On weak and viscosity solutions of nonlocal double phase equations}

\author[Y. Fang and C. Zhang  \hfil \hfilneg]{Yuzhou Fang and  Chao Zhang$^*$}

\thanks{$^*$ Corresponding author.}

\address{Yuzhou Fang \hfill\break School of Mathematics, Harbin Institute of Technology, Harbin 150001, China}
\email{18b912036@hit.edu.cn}

\address{Chao Zhang  \hfill\break School of Mathematics and Institute for Advanced Study in Mathematics, Harbin Institute of Technology, Harbin 150001, China}
\email{czhangmath@hit.edu.cn}


\maketitle

\begin{abstract}
We consider the nonlocal double phase equation
\begin{align*}
\mathrm{P.V.} &\int_{\mathbb{R}^n}|u(x)-u(y)|^{p-2}(u(x)-u(y))K_{sp}(x,y)\,dy\\
&+\mathrm{P.V.} \int_{\mathbb{R}^n} a(x,y)|u(x)-u(y)|^{q-2}(u(x)-u(y))K_{tq}(x,y)\,dy=0,
\end{align*}
where $1<p\leq q$ and the modulating coefficient $a(\cdot,\cdot)\geq0$. Under some suitable hypotheses, we first use the De Giorgi-Nash-Moser methods to derive the local H\"{o}lder continuity for bounded weak solutions, and then establish the relationship between weak solutions and viscosity solutions to such equations.
\end{abstract}

\section{Introduction}
\label{sec-1}

In this paper, we are concerned with the following nonlocal double phase problem
\begin{equation}
\label{main}
\mathcal{L}u=0  \quad\text{in } \Omega,
\end{equation}
where $\Omega$ is a bounded domain in $\mathbb{R}^n$ and the integro-differential operator $\mathcal{L}$ is defined as
\begin{align*}
\mathcal{L}u(x)=\mathrm{P.V.} &\int_{\mathbb{R}^n} |u(x)-u(y)|^{p-2}(u(x)-u(y))K_{sp}(x,y)\,dy\\
&+\mathrm{P.V.} \int_{\mathbb{R}^n} a(x,y)|u(x)-u(y)|^{q-2}(u(x)-u(y))K_{tq}(x,y)\,dy
\end{align*}
with $1<p\leq q$  and $a(\cdot,\cdot)\geq0$. Eq. \eqref{main} is a class of possibly degenerate and singular integro-differential equations switching between two diverse fractional elliptic phases according to the zero set of the modulating coefficient $a=a(\cdot,\cdot)$. Here the kernels $K_{sp},K_{tq}: \mathbb{R}^n\times\mathbb{R}^n\rightarrow(0,\infty)$ are symmetric measurable functions with differentiability orders $s,t\in(0,1)$ and summability exponents $p,q\in(1,\infty)$, respectively. The symbol P.V. means ``in the principal value sense".

Eq. \eqref{main} can be regarded naturally as the nonlocal counterpart to the classical double phase problem, whose representative model is closely connected with the following functional
\begin{equation}
\label{1-2}
u\mapsto \int(|Du|^p+a(x)|Du|^q)\,dx,  \quad 1<p\leq q, \quad a(x)\geq 0.
\end{equation}
This kind of functionals, firstly introduced by Zhikov \cite{Zhi93,Zhi95} in the setting of homogenization and Lavrentiev phenomenon, could provide useful models to formulate the behaviour of strongly anisotropic materials whose hardening properties change drastically with the point. The functionals possessing the non-uniform growth conditions,
$$
u\mapsto\int_\Omega F(x,u,D u)\,dx, \quad \nu|\xi|^p\leq F(x,u,\xi)\leq L(|\xi|^q+1),
$$
have been a surge of interest over the last decades. For the autonomous case that energy density $F(x,u,D u)\equiv F(D u)$, the regularity theory is well-known by the seminal papers of Marcellini \cite{Mar89,Mar91,Mar96}. Recently, the regularity issues for the double phase functionals have been explored in a series of papers by Colombo, Mingione et al. We refer the readers to \cite{CM15,CM215,BCM18}  for the $C^{1,\alpha}$ theory,
 \cite{BO17,CM16,DeFM} for the Calder\'{o}n-Zygmund estimates, \cite{CDeF20} for the obstacle problem, \cite{CZ20} for the potential theory, \cite{FZ20} for the equivalence between distributional and viscosity solutions and \cite{BBO21,FVZZ20,DeFO19} for the multi-phase problems. For more results, one can see for instance \cite{DeFM202,DeFM21,BBO20} and references therein.

For what concerns the nonlocal version of double phase problem, when $a(\cdot,\cdot)\equiv0$ the problem \eqref{main} is reduced to the celebrated fractional $p$-Laplace equation:
\begin{equation}
\label{1-3}
\mathrm{P.V.} \int_{\mathbb{R}^n} |u(x)-u(y)|^{p-2}(u(x)-u(y))K_{sp}(x,y)\,dy=0.
\end{equation}
This type of equations was initially considered by Ishii and Nakamura \cite{IN10}, in which they investigated the existence, uniqueness and convergence of viscosity solutions. When it comes to regularity theory for Eq. \eqref{1-3}, Di Castro, Kussi and Palatucci \cite{DKP16} showed the local boundedness and H\"{o}lder continuity for the weak solutions to \eqref{1-3}, in the spirit of De Giorgi-Nash-Moser theory; see also \cite{DKP14} for the nonlocal Harnack type inequalities. Subsequently, the H\"{o}lder regularity up to the boundary was established in \cite{IMS16}. Additionally, many other aspects of fractional $p$-Laplace type equations have already been studied: higher regularity \cite{BL17,BLS18,KMS15}, H\"{o}lder continuity of viscosity solutions \cite{Lin16}, fractional $p$-eigenvalue problems \cite{BP15,FP14} as well as the maximal principles and symmetry of solutions \cite{CL18}. More results can be found in \cite{ILPS15,KKP16,KMS152,Pal18,Sil06} and references therein.

The nonlocal double phase equation \eqref{main}  was introduced by De Filippis and Palatucci \cite{DeFP19}. For the inhomogeneous analogue
\begin{equation}
\label{1-4}
\mathcal{L}u=f,
\end{equation}
the authors in \cite{DeFP19} proved that any bounded viscosity solution is locally H\"{o}lder continuous under some reasonable hypotheses. This is the first regularity result for nonlocal double phase problems. Moreover, the self-improving properties for Eq. \eqref{1-4} were established by Scott and Mengesha in \cite{SM20}. To our knowledge, there are few results on the nonlocal double phase problems except the aforementioned two papers. To this end, our interest in the present article focuses on the H\"{o}lder regularity for weak solutions and the relationship between weak and viscosity solutions to \eqref{main}.

 Now we state the first result of this work as follows.

\begin{theorem}
\label{thm0-1}
Let $u$ be a bounded weak solution to \eqref{main} in $\Omega$. Under the assumptions {\rm ($A_1$), ($A_2$), ($H_1$), ($H_2$)} and \eqref{2-1} (in Section \ref{sec-2}), we infer that $u$ is locally H\"{o}lder continuous in  $\Omega$. Specifically, there exist two constants $\alpha\in\left(0,\frac{tq}{q-1}\right)$ and $C>0$, both of which depend on $n,p,q,s,t,\Lambda_1,\Lambda_2$ and $M$, such that
$$
\mathrm{osc}_{B_\rho(x_0)}u\leq C\left(\frac{\rho}{r}\right)^\alpha\|u\|_{L^\infty(\mathbb{R}^n)},
$$
where $\rho\in(0,r]$ and $B_{2r}(x_0)\subset\Omega$.
\end{theorem}

On the other hand, influenced by the papers \cite{DeFP19,KKL19,FZ20}, we try to consider the linkage between weak and viscosity solutions to \eqref{main}, which is the second result of our paper.

\begin{theorem}
\label{thm0-2}
Let the assumptions {\rm ($A_1$)--($A_4$), ($H_1'$), ($H_2$)--($H_4$)} and \eqref{2-1} (in Section \ref{sec-2}) be in force. Then the bounded weak solutions to \eqref{main}  are the viscosity solutions.
\end{theorem}

We would like to remark that the proof of Theorem \ref{thm0-1}  is inspired by the ideas developed in \cite{DKP16}. However, compared with the usual fractional $p$-Laplace equation, Eq. \eqref{main} exhibits the differences not only from the nonlocal feature of the involved integro-differential operators, but also from the non-standard growth behaviour and the presence of modulating coefficient $a(\cdot,\cdot)$. This makes the current investigation is more challenging. We have to pick an appropriate test function in order to establish the Logarithmic type lemma (Lemma \ref{lem3-2}), which plays a key role in the proof of H\"{o}lder regularity. We also need to take into account the barrier created by coefficient $a(\cdot,\cdot)$ in a suitable way to obtain the oscillation reduction. It is worth mentioning that the bound \eqref{2-1} used here is the same as that of \cite{DeFP19} for the H\"{o}lder regularity of viscosity solutions in homogeneous case. In addition, although the different notions of solutions to fractional $p$-Laplacian \eqref{1-3} have been investigated by Korvenp\"{a}\"{a}  et al. in \cite{KKL19,KKP17}; see also \cite{BM20} for the non-homogeneous version and \cite{FZ20} for the double phase case.  However, whether or not the different solutions  to \eqref{main}  coincide was still unknown. In this paper we  partially answer this question and establish that the weak solutions to \eqref{main} are  viscosity solutions  (Theorem \ref{thm0-2}). Unfortunately,  for the reverse implication, there exists a very tricky problem hindering in the proof. We shall continue this issue in a forthcoming paper.

This paper is organized as follows. In Section \ref{sec-2}, we give some basic notations and auxiliary tools to be used later as well as the definitions of solutions. Section \ref{sec-3} is devoted to establishing the H\"{o}lder estimates for bounded weak solutions. At last, we shall prove that bounded weak solutions are viscosity solutions in Section \ref{sec-4}.

\section{Preliminaries}
\label{sec-2}

In this section, we shall state some assumptions on the problem \eqref{main}, and give some basic notions and notations.

In the sequel, we denote by $C$ a generic positive constant which may vary from line to line. Relevant dependencies on parameters shall be emphasised utilizing parentheses, i.e., $C\equiv C(n,p,q)$ means that $C$ depends on $n,p,q$. Let
$$
B_r(x_0):=\{x\in\mathbb{R}^n:|x-x_0|<r\}
$$
denote the open ball with center $x_0$ and radius $r>0$. If not important, or clear from the context, we will not denote the center as follows: $B_r:=B_r(x_0)$. Moreover, let $\gamma B_r:=B(x_0,\gamma r)$. If $g\in L^1(A)$ and $A\subset\mathbb{R}^n$ is a measurable subset with positive measure $0<|A|<\infty$, we denote by
$$
(g)_A:=\fint_Ag(x)\,dx=\frac{1}{|A|}\int_Ag(x)\,dx
$$
its integral average.

Let the kernels $K_{sp}(\cdot,\cdot),K_{tq}(\cdot,\cdot): \mathbb{R}^n\times\mathbb{R}^n\rightarrow(0,\infty)$ be two measurable functions satisfying the following conditions:

\smallskip

\begin{itemize}

\item[($A_1$)] Symmetry: $K_{sp}(x,y)=K_{sp}(y,x), K_{tq}(x,y)=K_{tq}(y,x)$ for all $x,y\in\mathbb{R}^n$.

\smallskip

\item[($A_2$)] Growth condition: $\Lambda_1^{-1}\leq K_{sp}(x,y)|x-y|^{n+sp}\leq \Lambda_1$, $\Lambda_2^{-1}\leq K_{tq}(x,y)|x-y|^{n+tq}\leq \Lambda_2$ for all $x,y\in\mathbb{R}^n, x\neq y$, where $\Lambda_1,\Lambda_2\geq1$.

    \smallskip

\item[($A_3$)] Translation invariance: $K_{sp}(x+z,y+z)=K_{sp}(x,y)$, $K_{tq}(x+z,y+z)=K_{tq}(x,y)$ for all $x,y,z\in\mathbb{R}^n$, $x\neq y$.

    \smallskip

\item[($A_4$)] Continuity: the map $x\mapsto K_{sp}(x,y)$ is continuous in $\mathbb{R}^n\setminus\{y\}$ and $K_{tq}(x,y)$ has the same continuity property.

\end{itemize}
We then impose four conditions on the coefficient $a(\cdot,\cdot)$:

\smallskip

\begin{itemize}

\item[($H_1$)] Boundedness: $0\leq a(x,y)\leq M$ for $x,y\in\mathbb{R}^n$.

\smallskip

\item[($H_2$)] Symmetry: $a(x,y)=a(y,x)$ for all $x,y\in\mathbb{R}^n$.

\smallskip

\item[($H_3$)] Translation invariance: $a(x+z,y+z)=a(x,y)$ for all $x,y,z\in\mathbb{R}^n$.

\smallskip

\item[($H_4$)] Continuity: the map $x\mapsto a(x,y)$ is continuous in $\mathbb{R}^n$.
\end{itemize}
Throughout this paper we also assume that
\begin{equation}
\label{2-1}
1<p\leq q<\infty, \quad tq\leq sp.
\end{equation}

Let us point out that it suffices to require that the conditions \eqref{2-1}, ($A_1$), ($A_2$), ($H_1$) and ($H_2$) hold, when we prove the H\"{o}lder continuity of weak solutions. However, when we verify that weak solutions are viscosity solutions, we need the assumptions \eqref{2-1}, ($A_1$)--($A_4$) and ($H_2$)--($H_4$), together with the following stronger assumption that
\begin{itemize}
	\item [($H_1'$)] Positive boundedness:  $0<a(x,y)\leq M$ for  $x,y\in\mathbb{R}^n$.
\end{itemize}

The fractional Sobolev space $W^{s,p}(\mathbb{R}^n)$ with $s\in(0,1), p\in[1,\infty)$  is defined as
\begin{equation*}
W^{s,p}(\mathbb{R}^n)=\left\{u\in L^p(\mathbb{R}^n): [u]_{W^{s,p}(\mathbb{R}^n)}:=\int_{\mathbb{R}^n}\int_{\mathbb{R}^n}\frac{|u(x)-u(y)|^p}{|x-y|^{n+sp}}\,dx\,dy<\infty\right\},
\end{equation*}
endowed with the norm
$$
\|u\|_{W^{s,p}(\mathbb{R}^n)}:=\|u\|_{L^p(\mathbb{R}^n)}+[u]_{W^{s,p}(\mathbb{R}^n)}.
$$
In a similar way, it is possible to define the space $W^{s,p}(\Omega)$ in a region $\Omega\subset\mathbb{R}^n$.

The ``tail space" is given by
$$
L^{p-1}_{sp}(\mathbb{R}^n)=\left\{u\in L^{p-1}_{\rm loc}(\mathbb{R}^n): \int_{\mathbb{R}^n}\frac{|u(y)|^{p-1}}{(1+|y|)^{n+sp}}\,dy<\infty\right\}.
$$
The corresponding nonlocal tail of a function $u$ is defined as
$$
\mathrm{Tail}(u;z,r)=\left(r^{sp}\int_{\mathbb{R}^n\setminus B_r(z)}\frac{|u(y)|^{p-1}}{|y-z|^{n+sp}}\,dy\right)^\frac{1}{p-1},
$$
which is introduced in \cite{DKP16}. Analogously, taking into account the integro-differential operator whose kernel $K_{tq}(\cdot,\cdot)$ is perturbed by coefficient $a$, here we introduce a ``tail space with weight" and the corresponding nonlocal tail with weight denoted by
$$
L^{q-1}_{a,tq}(\mathbb{R}^n)=\left\{u\in L^{q-1}_{\rm loc}(\mathbb{R}^n): \sup_{x\in\mathbb{R}^n}\int_{\mathbb{R}^n}a(x,y)\frac{|u(y)|^{q-1}}{(1+|y|)^{n+tq}}\,dy<\infty\right\}
$$
and
$$
\mathrm{Tail}_{a}(u;z,r)=\left(r^{tq}\sup_{x\in\mathbb{R}^n}\int_{\mathbb{R}^n\setminus B_r(z)}a(x,y)\frac{|u(y)|^{q-1}}{|y-z|^{n+tq}}\,dy\right)^\frac{1}{q-1}.
$$
We can readily verify that $\mathrm{Tail}_a(u;z,r)$ is finite for every $z\in\mathbb{R}^n$ and $r\in(0,+\infty)$, provided that $u\in L^{q-1}_{a,tq}(\mathbb{R}^n)$.

Next let us recall the definition of weak solutions to \eqref{main}.

\begin{definition}
A function $u\in W^{s,p}(\mathbb{R}^n) \cap L^{p-1}_{sp}(\mathbb{R}^n)\cap L^{q-1}_{a,tq}(\mathbb{R}^n)$ is called a weak solution to Eq. \eqref{main}, if 
\begin{equation}
\label{2-2}
\int_{\mathbb{R}^n}\int_{\mathbb{R}^n}\left(|u(x)-u(y)|^pK_{sp}(x,y)+a(x,y)|u(x)-u(y)|^qK_{tq}(x,y)\right)\,dx\,dy<\infty
\end{equation}
and moreover
\begin{equation}
\label{2-3}
\begin{split}
\int_{\mathbb{R}^n}\int_{\mathbb{R}^n}\Big[&|u(x)-u(y)|^{p-2}(u(x)-u(y))(\phi(x)-\phi(y))K_{sp}(x,y)\\
&+a(x,y)|u(x)-u(y)|^{q-2}(u(x)-u(y))(\phi(x)-\phi(y))K_{tq}(x,y)\Big]\,dx\,dy=0,
\end{split}
\end{equation}
for all $\phi\in C_0^\infty(\mathbb{R}^n)$. A function $u\in W^{s,p}(\mathbb{R}^n)\cap L^{p-1}_{sp}(\mathbb{R}^n)\cap L^{q-1}_{a,tq}(\mathbb{R}^n)$ is called a weak supersolution to Eq. \eqref{main}, if the inequality \eqref{2-2} holds true and moreover
\begin{equation}
\label{2-4}
\begin{split}
\int_{\mathbb{R}^n}\int_{\mathbb{R}^n}\Big[&|u(x)-u(y)|^{p-2}(u(x)-u(y))(\phi(x)-\phi(y))K_{sp}(x,y)\\
&+a(x,y)|u(x)-u(y)|^{q-2}(u(x)-u(y))(\phi(x)-\phi(y))K_{tq}(x,y)\Big]\,dx\,dy\geq0,
\end{split}
\end{equation}
for all nonnegative $\phi\in C_0^\infty(\mathbb{R}^n)$. The inequality \eqref{2-4} is reverse for subsolution.
\end{definition}

\begin{remark}
\label{rem2-1}
From the Theorem 2.3 in \cite{SM20}, we find that
the admissible test functions $\phi\in C_0^\infty(\mathbb{R}^n)$ in the previous definition can be replaced by $\phi\in W^{s,p}(\mathbb{R}^n)\cap L^\infty(\mathbb{R}^n)$ with compact support satisfying the inequality \eqref{2-2}.
\end{remark}

We denote the set of critical points of function $u$ by
$$
N_u=\{x\in\Omega: Du(x)=0\}.
$$
Let $d_u(x):=\mathrm{dist}(x,N_u)$ stand for the distance from point $x$ to set $N_u$. Given an open set $E\subset \Omega$, we define
$$
C^2_\beta(E):=\left\{u\in C^2(E):\sup_{x\in E}\left(\frac{\min\{d_u(x),1\}^{\beta-1}}{|Du(x)|}+\frac{|D^2u(x)|}{d_u(x)^{\beta-2}}\right)<\infty\right\}.
$$
In the spirit of \cite{KKL19}, we now give the notion of viscosity solutions to \eqref{main}.

\begin{definition}
\label{def1}
A function $u:\mathbb{R}^n\rightarrow[-\infty,+\infty]$ is called a viscosity supersolution to Eq. \eqref{main}, if it satisfies the following four properties:
\begin{itemize}
\item[(i)] $u$ is lower semicontinuous in $\Omega$.

\smallskip

\item[(ii)] $u<+\infty$ almost everywhere in $\mathbb{R}^n$, and $u>-\infty$ everywhere in $\Omega$.

\smallskip

\item[(iii)] If $\psi\in C^2(B_r(x_0))$ for some $B_r(x_0)\subset\Omega$ such that $\psi(x_0)=u(x_0)$ and $\psi(x)\leq u(x)$ in $B_r(x_0)$, and moreover one of the following holds
    \begin{itemize}

    \item[(a)] $p>\frac{2}{2-s}$ or $D\psi(x_0)\neq0$,

    \item[(b)] $1<p\leq\frac{2}{2-s}$, $D\psi(x_0)=0$ with $x_0$ being an isolated critical point of $\psi$, and $\psi\in C^2_\beta(B_r(x_0))$ for some $\beta>\frac{sp}{p-1}$,
    \end{itemize}
then one has
$$
\mathcal{L}\psi_r(x_0)\geq 0,
$$
where
\begin{equation*}
\psi_r(x)=\begin{cases}\psi(x),  & \text{\textmd{for }} x\in B_r(x_0),\\[2mm]
u(x), & \text{\textmd{for }} x\in \mathbb{R}^n\setminus B_r(x_0).
\end{cases}
\end{equation*}

\smallskip

\item[(iv)] $u_-\in L_{sp}^{p-1}(\mathbb{R}^n)\cap L_{a,tq}^{q-1}(\mathbb{R}^n)$.
\end{itemize}
If $-u$ is a viscosity supersolution, then we call $u$ a viscosity subsolution. A function $u$ is a viscosity solution if and only if it is viscosity super- and subsolution. Here $u_-:=\max\{-u,0\}$.
\end{definition}

For the sake of convenience, we provide two very important inequalities to be employed later.

\begin{proposition}[\cite{DKP16}]
\label{pro2-1}
Let $\kappa\geq1$ and $\varepsilon\in(0,1]$. Then, for all $a,b\in\mathbb{R}^n$ ($n\geq1$), it holds that
$$
|a|^\kappa\leq |b|^\kappa+c_\kappa\varepsilon|b|^\kappa+(1+c_\kappa\varepsilon)\varepsilon^{1-\kappa}|a-b|^\kappa,
$$
where $c_\kappa:=(\kappa-1)\Gamma(\max\{1,\kappa-2\})$ and $\Gamma$ denotes the standard Gamma function.
\end{proposition}

\begin{proposition}[\cite{KKL19}]
\label{pro2-2}
Let $\kappa>1$ and $a,b\in\mathbb{R}$. Then
$$
\left||a|^{\kappa-2}a-|b|^{\kappa-2}b\right|\leq C(|b|+|a-b|)^{\kappa-2}|a-b|,
$$
where $C$ depends only on $\kappa$.
\end{proposition}

\section{H\"{o}lder continuity of weak solutions}
\label{sec-3}

This section is devoted to showing the H\"{o}lder regularity of weak solutions to Eq. \eqref{main}. We start with obtaining the Caccioppoli type inequality in the nonlocal framework. Let the assumptions ($A_1$), ($A_2$), ($H_1$), ($H_2$) and \eqref{2-1} be in force. In the next lemma, set $v_+(x):=(u(x)-k)_+=\max\{u(x)-k,0\}$ and $v_-:=(u(x)-k)_-=(k-u(x))_+$, where $k\in\mathbb{R}$.

\begin{lemma} [Caccioppoli's inequality]
\label{lem3-1}
Let $1<p\leq q<\infty$. Assume that $u\in W^{s,p}(\mathbb{R}^n)\cap L^\infty_{\rm loc}(\mathbb{R}^n)$ is a weak solution to Eq. \eqref{main}. Then, for any $B_r(x_0)\subset\Omega$ and nonnegative $\phi\in C^\infty_0(B_r(x_0))$, there holds that
\begin{align}
\label{3-1}
&\quad\int_{B_r}\int_{B_r}\bigg(\left|v_\pm(x)\phi^\frac{q}{p}(x)-v_\pm(y)\phi^\frac{q}{p}(y)\right|^pK_{sp}(x,y)
+a(x,y)|v_\pm(x)\phi(x)-v_\pm(y)\phi(y)|^qK_{tq}(x,y)\bigg)\,dx\,dy  \nonumber\\
&\leq C\int_{B_r}\int_{B_r}\bigg((\max\{v_\pm(x),v_\pm(y)\})^p\left|\phi^\frac{q}{p}(x)-\phi^\frac{q}{p}(y)\right|^pK_{sp}(x,y) \nonumber\\
&\qquad\qquad\quad+a(x,y)(\max\{v_\pm(x),v_\pm(y)\})^q|\phi(x)-\phi(y)|^qK_{tq}(x,y)\bigg)\,dx\,dy  \nonumber\\
&\quad +C\int_{B_r}v_\pm(x)\phi^q(x)\,dx\left(\sup_{x\in \mathrm{supp}\,\phi}\int_{\mathbb{R}^n\setminus B_r}v_\pm^{p-1}(y)K_{sp}(x,y)+a(x,y)v_\pm^{q-1}(y)K_{tq}(x,y)\,dy\right),
\end{align}
where $C$ depends only on $p$ and $q$.
\end{lemma}

\begin{proof}
We just verify this claim for $v_+$. Let $\eta:=v_+\phi^q$ with $v_+(x):=\max\{u(x)-k,0\}$ and $0\leq\phi\in C^\infty_0(B_r(x_0))$. From Remark \ref{rem2-1}, we observe that $\eta$ can serve as a test function. Now we take $\eta$ to test the weak formulation of Eq. \eqref{main}, then
\begin{equation}
\label{3-2}
\begin{split}
0&\geq\int_{B_r}\int_{B_r}\Big(|u(x)-u(y)|^{p-2}(u(x)-u(y))(v_+(x)\phi^q(x)-v_+(y)\phi^q(y))K_{sp}(x,y)\\
&\quad\quad+a(x,y)|u(x)-u(y)|^{q-2}(u(x)-u(y))(v_+(x)\phi^q(x)-v_+(y)\phi^q(y))K_{tq}(x,y)\Big)\,dx\,dy\\
&\quad+2\int_{\mathbb{R}^n\setminus B_r}\int_{B_r}\Big(|u(x)-u(y)|^{p-2}(u(x)-u(y))v_+(x)\phi^q(x)K_{sp}(x,y)\\
&\quad\quad+a(x,y)|u(x)-u(y)|^{q-2}(u(x)-u(y))v_+(x)\phi^q(x)K_{tq}(x,y)\Big)\,dx\,dy\\
&=:I_1+2I_2.
\end{split}
\end{equation}
We first consider the integrand of $I_1$. With no loss of generality, we assume $u(x)\geq u(y)$ (otherwise exchange the roles of $x$ and $y$ below), then
\begin{align*}
&\quad |u(x)-u(y)|^{s-2}(u(x)-u(y))(v_+(x)\phi^q(x)-v_+(y)\phi^q(y))\\
&\geq (v_+(x)-v_+(y))^{s-1}(v_+(x)\phi^q(x)-v_+(y)\phi^q(y)),
\end{align*}
where $s\in\{p,q\}$. Thus, it follows that
\begin{align}
\label{3-3}
I_1&\geq \int_{B_r}\int_{B_r}\Big(|v_+(x)-v_+(y)|^{p-2}(v_+(x)-v_+(y))(v_+(x)\phi^q(x)-v_+(y)\phi^q(y))K_{sp}(x,y) \nonumber\\
&\qquad+a(x,y)|v_+(x)-v_+(y)|^{q-2}(v_+(x)-v_+(y))(v_+(x)\phi^q(x)-v_+(y)\phi^q(y))K_{tq}(x,y)\Big)\,dx\,dy.
\end{align}
For the integrand of $I_2$, we get
$$
|u(x)-u(y)|^{s-2}(u(x)-u(y))v_+(x)\geq-v_+^{s-1}(y)v_+(x)
$$
with $s\in\{p,q\}$, which implies that
\begin{equation}
\label{3-4}
\begin{split}
I_2&\geq-\int_{\mathbb{R}^n\setminus B_r}\int_{B_r}\Big(v_+^{p-1}(y)K_{sp}(x,y)+a(x,y)v_+^{q-1}(y)K_{tq}(x,y)\Big)v_+(x)\phi^q(x)\,dx\,dy\\
&\geq-\int_{B_r}v_+(x)\phi^q(x)\,dx\left(\sup_{x\in \mathrm{supp}\,\phi}\int_{\mathbb{R}^n\setminus B_r}\Big(v^{p-1}_+(y)K_{sp}(x,y)+a(x,y)v_+^{q-1}(y)K_{tq}(x,y)\Big)\,dy\right).
\end{split}
\end{equation}

We next deal with the integral in \eqref{3-3}. If $v_+(x)\geq v_+(y)$ and $\phi(y)\geq\phi(x)$, we apply Proposition \ref{pro2-1} to deduce that for every $\varepsilon\in(0,1]$,
$$
\phi^q(x)\geq(1-c_q\varepsilon)\phi^q(y)-(1+c_q\varepsilon)\varepsilon^{1-q}|\phi(x)-\phi(y)|^q.
$$
Then we choose
$$
\varepsilon=\frac{1}{\max\{1,2c_q\}}\frac{v_+(x)-v_+(y)}{v_+(x)}\in (0,1],
$$
which leads to
\begin{align*}
&\quad(v_+(x)-v_+(y))^{q-1}(v_+(x)\phi^q(x)-v_+(y)\phi^q(y))\\
&\geq \frac{1}{2}(v_+(x)-v_+(y))^q(\max\{\phi(x),\phi(y)\})^q-C(q)(\max\{v_+(x),v_+(y)\})^q|\phi(x)-\phi(y)|^q.
\end{align*}
Hence, in general cases we have
\begin{align*}
&\quad|v_+(x)-v_+(y)|^{q-2}(v_+(x)-v_+(y))(v_+(x)\phi^q(x)-v_+(y)\phi^q(y))\\
&\geq \frac{1}{2}|v_+(x)-v_+(y)|^q(\max\{\phi(x),\phi(y)\})^q-C(q)(\max\{v_+(x),v_+(y)\})^q|\phi(x)-\phi(y)|^q.
\end{align*}
On the other hand, in a similar way we derive
\begin{align*}
&\quad|v_+(x)-v_+(y)|^{p-2}(v_+(x)-v_+(y))(v_+(x)\phi^q(x)-v_+(y)\phi^q(y))\\
&\geq \frac{1}{2}|v_+(x)-v_+(y)|^p\left(\max\left\{\phi^\frac{q}{p}(x),\phi^\frac{q}{p}(y)\right\}\right)^p\\
&\quad-C(p)(\max\{v_+(x),v_+(y)\})^p\left|\phi^\frac{q}{p}(x)-\phi^\frac{q}{p}(y)\right|^p.
\end{align*}
Consequently, \eqref{3-3} becomes
\begin{align*}
I_1&\geq \frac{1}{2}\int_{B_r}\int_{B_r}\bigg[|v_+(x)-v_+(y)|^p\left(\max\left\{\phi^\frac{q}{p}(x),\phi^\frac{q}{p}(y)\right\}\right)^p
K_{sp}(x,y)\\
&\quad\quad\quad+a(x,y)|v_+(x)-v_+(y)|^q(\max\{\phi(x),\phi(y)\})^qK_{tq}(x,y)\bigg]\,dx\,dy\\
&\quad-C(p,q)\int_{B_r}\int_{B_r}\bigg[(\max\{v_+(x),v_+(y)\})^p\left|\phi^\frac{q}{p}(x)-\phi^\frac{q}{p}(y)\right|^pK_{sp}(x,y)\\
&\quad\qquad+a(x,y)(\max\{v_+(x),v_+(y)\})^q|\phi(x)-\phi(y)|^qK_{tq}(x,y)\bigg]\,dx\,dy.
\end{align*}
Notice that
\begin{align*}
|v_+(x)\phi(x)-v_+(y)\phi(y)|^s&\leq2^{s-1}|v_+(x)-v_+(y)|^s(\max\{\phi(x),\phi(y)\})^s\\
&\quad+2^{s-1}(\max\{v_+(x),v_+(y)\})^s|\phi(x)-\phi(y)|^s
\end{align*}
with $s\in\{p,q\}$. We further get
\begin{align}
\label{3-5}
I_1&\geq 2^{-(p+q)}\int_{B_r}\int_{B_r}\bigg[\left|v_+(x)\phi^\frac{q}{p}(x)-v_+(y)\phi^\frac{q}{p}(y)\right|^p
K_{sp}(x,y)  \nonumber\\
&\qquad\qquad\quad+a(x,y)|v_+(x)\phi(x)-v_+(y)\phi(y)|^qK_{tq}(x,y)\bigg]\,dx\,dy  \nonumber\\
&\quad-C(p,q)\int_{B_r}\int_{B_r}\bigg[(\max\{v_+(x),v_+(y)\})^p\left|\phi^\frac{q}{p}(x)-\phi^\frac{q}{p}(y)\right|^pK_{sp}(x,y) \nonumber\\
&\quad\qquad\qquad+a(x,y)(\max\{v_+(x),v_+(y)\})^q|\phi(x)-\phi(y)|^qK_{tq}(x,y)\bigg]\,dx\,dy.
\end{align}
Merging \eqref{3-5}, \eqref{3-4} with \eqref{3-2} leads to the desired result \eqref{3-1}. We now finish the proof.
\end{proof}

Next, we establish the second important tool, logarithmic estimate, which plays a key role in the proof of H\"{o}lder continuity. We state it as follows.

\begin{lemma}[Logarithmic lemma]
\label{lem3-2}
Let $1<p\leq q<\infty$. Suppose that $u\in W^{s,p}(\mathbb{R}^n)\cap L^\infty_{\rm loc}(\mathbb{R}^n)$ is a weak supersolution to Eq. \eqref{main} satisfying $u\geq0$ in $B_R(x_0)\subset\Omega$. Then for every $B_r(x_0)\subset B_{R/2}(x_0)$ and every $d>0$, one has
\begin{align}
&\quad\int_{B_r}\int_{B_r}\Bigg[\left|\log\left(\frac{u(x)+d}{u(y)+d}\right)\right|^pK_{sp}(x,y)
+d^{q-p}a(x,y)\left|\log\left(\frac{u(x)+d}{u(y)+d}\right)\right|^qK_{tq}(x,y)\Bigg]\,dx\,dy  \nonumber\\
&\leq Cd^{1-p}r^n\left(R^{-sp}[\mathrm{Tail}(u_-;x_0,R)]^{p-1}+R^{-tq}[\mathrm{Tail}_{a}(u_-;x_0,R)]^{q-1}\right) \nonumber\\
&\quad+Cr^n\left(r^{-sp}+M(\|u\|_{L^\infty(\Omega)}+d)^{q-p}r^{-tq}\right),
\end{align}
where $u_-=\max\{-u,0\}$ and $C$ depends only on $n,p,q,s,t,\Lambda_1,\Lambda_2$.
\end{lemma}

\begin{remark}
If $u\in L^\infty(\mathbb{R}^n)$, then we can see that
$$
[\mathrm{Tail}(u_-;x_0,R)]^{p-1}=R^{sp}\int_{\mathbb{R}^n\setminus B_R}\frac{u_-^{p-1}}{|x-x_0|^{n+sp}}\,dx\leq\|u\|^{p-1}_{L^\infty(\mathbb{R}^n)}
$$
and
$$
[\mathrm{Tail}_{a}(u_-;x_0,R)]^{q-1}=R^{tq}\sup_{x\in B_R}\int_{\mathbb{R}^n\setminus B_R}a(x,y)\frac{u_-^{q-1}}{|x-x_0|^{n+tq}}\,dx\leq M\|u\|^{q-1}_{L^\infty(\mathbb{R}^n)}.
$$
\end{remark}

\begin{proof}
We take a test function $\eta(x)$ as
$$
\eta(x):=(u(x)+d)^{1-p}\phi^q(x),
$$
where $\phi\in C^\infty_0(B_{3r/2})$ is such that $0\leq\phi\leq1$, $\phi\equiv1$ in $B_r$ and $|D\phi|\leq Cr^{-1}$ in $B_{3r/2}$. Obviously, $u\ge 0$ in the support of $\phi$. Now we have
\begin{align}
\label{3-6}
0&\le \int_{\mathbb{R}^n}\int_{\mathbb{R}^n}\Big(|u(x)-u(y)|^{p-2}(u(x)-u(y))(\eta(x)-\eta(y))K_{sp}(x,y)   \nonumber\\
&\qquad\qquad+a(x,y)|u(x)-u(y)|^{q-2}(u(x)-u(y))(\eta(x)-\eta(y))K_{tq}(x,y)\Big)\,dx\,dy  \nonumber\\
&=\int_{B_{2r}}\int_{B_{2r}}\bigg[|u(x)-u(y)|^{p-2}(u(x)-u(y))\left(\frac{\phi^q(x)}{(u(x)+d)^{p-1}}-\frac{\phi^q(y)}{(u(y)+d)^{p-1}}\right)K_{sp}(x,y) \nonumber\\
&\qquad\qquad+a(x,y)|u(x)-u(y)|^{q-2}(u(x)-u(y))\left(\frac{\phi^q(x)}{(u(x)+d)^{p-1}}-\frac{\phi^q(y)}{(u(y)+d)^{p-1}}\right)K_{tq}(x,y)\bigg]\,dx\,dy \nonumber\\
&\quad+2\int_{\mathbb{R}^n\setminus B_{2r}}\int_{B_{2r}}\bigg[|u(x)-u(y)|^{p-2}(u(x)-u(y))\frac{\phi^q(x)}{(u(x)+d)^{p-1}}K_{sp}(x,y)  \nonumber\\
&\quad\qquad\quad+a(x,y)|u(x)-u(y)|^{q-2}(u(x)-u(y))\frac{\phi^q(x)}{(u(x)+d)^{p-1}}K_{tq}(x,y)\bigg]\,dx\,dy  \nonumber\\
&=:I_1+I_2.
\end{align}
We are ready to estimate the integral $I_1$ under the condition that $u(x)>u(y)$. We consider the first term of the integrand of $I_1$. Let
$$
\varphi(x):=\phi^\frac{q}{p}(x) \quad\text{and} \quad \varepsilon:=\delta\frac{u(x)-u(y)}{u(x)+d}\in(0,1)
$$
with $\delta\in(0,1)$. It follows from Proposition \ref{pro2-1} that
\begin{align*}
&\quad |u(x)-u(y)|^{p-2}(u(x)-u(y))\left(\frac{\phi^q(x)}{(u(x)+d)^{p-1}}-\frac{\phi^q(y)}{(u(y)+d)^{p-1}}\right)\\
&\leq \left(\frac{u(x)-u(y)}{u(x)+d}\right)^p\varphi^p(x)\left[\frac{1-\left(\frac{u(y)+d}{u(x)+d}\right)^{1-p}}{1-\frac{u(y)+d}{u(x)+d}}
+c_p\delta\right]+c_p\delta^{1-p}|\varphi(x)-\varphi(y)|^p\\
&=:I_{1,1}+c_p\delta^{1-p}|\varphi(x)-\varphi(y)|^p.
\end{align*}
If $u(y)+d\le \frac{u(x)+d}{2}$, then
$$
I_{1,1}\leq \left(c_p\delta-\frac{p-1}{2^p}\right)\left(\frac{u(x)-u(y)}{u(y)+d}\right)^{p-1}\varphi^p(y).
$$
Choosing $\delta:=\frac{p-1}{2^{p+1}c_p}$ yields that
$$
I_{1,1}\leq -\frac{p-1}{2^{p+1}}\left(\frac{u(x)-u(y)}{u(y)+d}\right)^{p-1}\varphi^p(y).
$$
If $u(y)+d>\frac{u(x)+d}{2}$, we get
$$
I_{1,1}\leq [c_p\delta-(p-1)]\left(\frac{u(x)-u(y)}{u(x)+d}\right)^p\varphi^p(y),
$$
by the choice of $\delta$. Then we have
$$
I_{1,1}\leq -\frac{(p-1)\left(2^{p+1}-1\right)}{2^{p+1}}\left(\frac{u(x)-u(y)}{u(x)+d}\right)^p\varphi^p(y).
$$
Thus via the elementary inequalities
\begin{equation}
\label{3-6-1}
\begin{cases}(\log t)^p\leq c(t-1)^{p-1},  & \text{\textmd{if }} t>2,\\[2mm]
\log(1+t)\leq t,  & \text{\textmd{if }} t\ge0,
\end{cases}
\end{equation}
we derive
\begin{align}
\label{3-7}
&\quad |u(x)-u(y)|^{p-2}(u(x)-u(y))\left(\frac{\phi^q(x)}{(u(x)+d)^{p-1}}-\frac{\phi^q(y)}{(u(y)+d)^{p-1}}\right) \nonumber\\
&\leq -\frac{1}{C(p)}\left[\log\left(\frac{u(x)+d}{u(y)+d}\right)\right]^p\varphi^p(y)+C(p)\delta^{1-p}|\varphi(x)-\varphi(y)|^p.
\end{align}
We proceed with the second term of the integrand of $I_1$. Using Proposition \ref{pro2-1} again, we arrive at
\begin{align}
\label{3-8}
&\quad |u(x)-u(y)|^{q-2}(u(x)-u(y))\left(\frac{\phi^q(x)}{(u(x)+d)^{p-1}}-\frac{\phi^q(y)}{(u(y)+d)^{p-1}}\right) \nonumber\\
&\leq (u(x)-u(y))^{q-1}\frac{\phi^q(y)+c_q\delta\frac{u(x)-u(y)}{u(x)+d}\phi^q(y)+(1+c_q)\delta^{1-q}\left(\frac{u(x)-u(y)}{u(x)+d}\right)^{1-q}|\phi(x)-\phi(y)|^q}{(u(x)+d)^{p-1}} \nonumber\\
&\quad-(u(x)-u(y))^{q-1}\frac{\phi^q(y)}{(u(y)+d)^{p-1}} \nonumber\\
&=\left(\frac{u(x)-u(y)}{u(x)+d}\right)^{q-1}\phi^q(y)\left[1+c_q\delta\frac{u(x)-u(y)}{u(x)+d}-\left(\frac{u(x)+d}{u(y)+d}\right)^{p-1}\right](u(x)+d)^{q-p} \nonumber\\
&\quad+(1+c_q)\delta^{1-q}|\phi(x)-\phi(y)|^q(u(x)+d)^{q-p} \nonumber\\
&=\left(\frac{u(x)-u(y)}{u(x)+d}\right)^{q-1}\phi^q(y)\left[\frac{1-\left(\frac{u(y)+d}{u(x)+d}\right)^{1-p}}{1-\frac{u(y)+d}{u(x)+d}}
+c_q\delta\right](u(x)+d)^{q-p}  \nonumber\\
&\quad+(1+c_q)\delta^{1-q}|\phi(x)-\phi(y)|^q(u(x)+d)^{q-p}   \nonumber\\
&=:I_{1,2}+(1+c_q)\delta^{1-q}|\phi(x)-\phi(y)|^q(u(x)+d)^{q-p}.
\end{align}
When $u(y)+d\le \frac{u(x)+d}{2}$, it follows that
$$
I_{1,2}\leq \left(c_q\delta-\frac{p-1}{2^p}\right)\left(\frac{u(x)-u(y)}{u(y)+d}\right)^{q-1}\phi^q(y)(u(x)+d)^{q-p}.
$$
By selecting $\delta:=\frac{p-1}{2^{p+1}c_q}$ we have
\begin{equation}
\label{3-8-1}
I_{1,2}\leq -\frac{p-1}{2^{p+1}}(u(x)+d)^{q-p}\left(\frac{u(x)-u(y)}{u(y)+d}\right)^{q-1}\phi^q(y).
\end{equation}
In the case $\frac{u(y)+d}{u(x)+d}\in (\frac{1}{2},1)$, we obtain
$$
I_{1,2}\leq [c_q\delta-(p-1)](u(x)+d)^{q-p}\left(\frac{u(x)-u(y)}{u(x)+d}\right)^q\phi^q(y),
$$
and further
\begin{equation}
\label{3-8-2}
I_{1,2}\leq -\frac{(p-1)\left(2^{p+1}-1\right)}{2^{p+1}}(u(x)+d)^{q-p}\left(\frac{u(x)-u(y)}{u(x)+d}\right)^q\phi^q(y).
\end{equation}
We now apply \eqref{3-6-1} to infer that
\begin{equation}
\label{3-9}
\left[\log\left(\frac{u(x)+d}{u(y)+d}\right)\right]^q\leq C\left(\frac{u(x)-u(y)}{u(y)+d}\right)^{q-1} \quad \text{for } \frac{u(y)+d}{u(x)+d}<\frac{1}{2}
\end{equation}
and
\begin{equation}
\label{3-10}
\left[\log\left(\frac{u(x)+d}{u(y)+d}\right)\right]^q\leq C\left(\frac{u(x)-u(y)}{u(x)+d}\right)^q \quad \text{for } \frac{1}{2}\leq\frac{u(y)+d}{u(x)+d}<1.
\end{equation}
Thus, putting together \eqref{3-8-1}--\eqref{3-10} derives that
\begin{align}
\label{3-11}
I_{1,2}&\leq -\frac{1}{C(p,q)}(u(x)+d)^{q-p}\left[\log\left(\frac{u(x)+d}{u(y)+d}\right)\right]^q\phi^q(y) \nonumber\\
&\leq -\frac{1}{C(p,q)}d^{q-p}\left[\log\left(\frac{u(x)+d}{u(y)+d}\right)\right]^q\phi^q(y).
\end{align}
It follows from \eqref{3-8} and \eqref{3-11} that
\begin{align*}
&\quad |u(x)-u(y)|^{q-2}(u(x)-u(y))\left(\frac{\phi^q(x)}{(u(x)+d)^{p-1}}-\frac{\phi^q(y)}{(u(y)+d)^{p-1}}\right)\\
&\leq -\frac{1}{C}d^{q-p}\left[\log\left(\frac{u(x)+d}{u(y)+d}\right)\right]^q\phi^q(y)+C(\|u\|_{L^\infty(\Omega)}+d)^{q-p}|\phi(x)-\phi(y)|^q
\end{align*}
If $u(y)>u(x)$, then we can exchange the roles of $x$ and $y$ in the preceding calculations. Consequently,
\begin{align}
\label{3-12}
I_1 &\leq -\frac{1}{C}\int_{B_{2r}}\int_{B_{2r}}\Bigg[\left|\log\left(\frac{u(x)+d}{u(y)+d}\right)\right|^pK_{sp}(x,y)\phi^q(y) \nonumber\\
&\qquad\qquad\qquad+d^{q-p}a(x,y)\left|\log\left(\frac{u(x)+d}{u(y)+d}\right)\right|^qK_{tq}(x,y)\phi^q(y)\Bigg]\,dx\,dy  \nonumber\\
&\quad+C\int_{B_{2r}}\int_{B_{2r}}\Big[\left|\phi^\frac{q}{p}(x)-\phi^\frac{q}{p}(y)\right|^pK_{sp}(x,y) \nonumber\\
&\qquad\qquad\qquad+(\|u\|_{L^\infty(\Omega)}+d)^{q-p}a(x,y)|\phi(x)-\phi(y)|^qK_{tq}(x,y)\Big]\,dx\,dy.
\end{align}

Next, we deal with the integral $I_2$. We can easily evaluate
\begin{align*}
I_2&=2\int_{B_R\setminus B_{2r}}\int_{B_{2r}}\bigg[|u(x)-u(y)|^{p-2}(u(x)-u(y))\frac{\phi^q(x)}{(u(x)+d)^{p-1}}K_{sp}(x,y) \\
&\qquad\qquad\quad+a(x,y)|u(x)-u(y)|^{q-2}(u(x)-u(y))\frac{\phi^q(x)}{(u(x)+d)^{p-1}}K_{tq}(x,y)\bigg]\,dx\,dy\\
&\quad+2\int_{\mathbb{R}^n\setminus B_R}\int_{B_{2r}}\bigg[|u(x)-u(y)|^{p-2}(u(x)-u(y))\frac{\phi^q(x)}{(u(x)+d)^{p-1}}K_{sp}(x,y)  \\
&\qquad\qquad\quad+a(x,y)|u(x)-u(y)|^{q-2}(u(x)-u(y))\frac{\phi^q(x)}{(u(x)+d)^{p-1}}K_{tq}(x,y)\bigg]\,dx\,dy\\
&\leq C\int_{\mathbb{R}^n\setminus B_{2r}}\int_{B_{2r}}\left(K_{sp}(x,y)+\|u\|_{L^\infty(\Omega)}^{q-p}a(x,y)K_{tq}(x,y)\right)\phi^q(x)\,dx\,dy\\
&\quad+C\int_{\mathbb{R}^n\setminus B_R}\int_{B_{2r}}d^{1-p}\left[(u(y))^{p-1}_-K_{sp}(x,y)+a(x,y)(u(y))^{q-1}_-K_{tq}(x,y)\right]\phi^q(x)\,dx\,dy,
\end{align*}
where $C$ depends only on $p$ and $q$. Exploiting the assumptions ($A_2$) and ($H_1$), as well as the fact that $\mathrm{supp}\,\phi\subset\subset B_{{3r}/{2}}$, we conclude the following estimates,
\begin{align*}
&\quad\int_{\mathbb{R}^n\setminus  B_{2r}}\int_{B_{2r}}\left(K_{sp}(x,y)+\|u\|_{L^\infty(\Omega)}^{q-p}a(x,y)K_{tq}(x,y)\right)\phi^q(x)\,dx\,dy\\
&\leq Cr^{n-sp}+CM\|u\|_{L^\infty(\Omega)}^{q-p}r^{n-tq},
\end{align*}

\begin{align*}
\int_{\mathbb{R}^n\setminus B_R}\int_{B_{2r}}(u(y))^{p-1}_-K_{sp}(x,y)\phi^q(x)\,dx\,dy\leq C\frac{r^n}{R^{sp}}[\mathrm{Tail}(u_-;x_0,R)]^{p-1}
\end{align*}
and
\begin{align*}
&\quad\int_{\mathbb{R}^n\setminus B_R}\int_{B_{2r}}a(x,y)(u(y))^{q-1}_-K_{tq}(x,y)\phi^q(x)\,dx\,dy\\
&\leq \int_{\mathbb{R}^n\setminus B_R}\int_{B_\frac{3r}{2}}a(x,y)(u(y))^{q-1}_-\frac{\Lambda_2}{|x-y|^{n+tq}}\,dx\,dy\\
&\leq C\int_{\mathbb{R}^n\setminus B_R}\int_{B_\frac{3r}{2}}a(x,y)\frac{(u(y))^{q-1}_-}{|x_0-y|^{n+tq}}\,dx\,dy\\
&\leq C\frac{r^n}{R^{tq}}\left(R^{tq}\sup_{x\in B_R(x_0)}\int_{\mathbb{R}^n\setminus B_R(x_0)}a(x,y)\frac{(u(y))^{q-1}_-}{|x_0-y|^{n+tq}}\,dy\right)\\
&=:C\frac{r^n}{R^{tq}}[\mathrm{Tail}_a(u_-;x_0,R)]^{q-1}.
\end{align*}

Therefore,
\begin{align}
\label{3-13}
I_2&\leq Cd^{1-p}\left(\frac{r^n}{R^{sp}}[\mathrm{Tail}(u_-;x_0,R)]^{p-1}+\frac{r^n}{R^{tq}}[\mathrm{Tail}_a(u_-;x_0,R)]^{q-1}\right) \nonumber\\
&\quad+C\left(r^{n-sp}+M\|u\|_{L^\infty(\Omega)}^{q-p}r^{n-tq}\right).
\end{align}
Merging the displays \eqref{3-6}, \eqref{3-12}, \eqref{3-13}, we deduce that
\begin{align*}
&\quad\int_{B_{2r}}\int_{B_{2r}}\left[\left|\log\left(\frac{u(x)+d}{u(y)+d}\right)\right|^pK_{sp}(x,y)
+d^{q-p}a(x,y)\left|\log\left(\frac{u(x)+d}{u(y)+d}\right)\right|^qK_{tq}(x,y)\right]\phi^q(y)\,dx\,dy\\
&\leq C\int_{B_{2r}}\int_{B_{2r}}\Big[\left|\phi^\frac{q}{p}(x)-\phi^\frac{q}{p}(y)\right|^pK_{sp}(x,y)
+a(x,y)\frac{|\phi(x)-\phi(y)|^q}{(\|u\|_{L^\infty(\Omega)}+d)^{p-q}}K_{tq}(x,y)\Big]\,dx\,dy\\
&\quad +Cd^{1-p}\left(\frac{r^n}{R^{sp}}[\mathrm{Tail}(u_-;x_0,R)]^{p-1}+\frac{r^n}{R^{tq}}[\mathrm{Tail}_a(u_-;x_0,R)]^{q-1}\right) \nonumber\\
&\quad+C\left(r^{n-sp}+M\|u\|_{L^\infty(\Omega)}^{q-p}r^{n-tq}\right).
\end{align*}
By the mean value theorem,
$$
|\phi(x)-\phi(y)|^q\leq Cr^{-q}|x-y|^q
$$
and
$$
\left|\phi^\frac{q}{p}(x)-\phi^\frac{q}{p}(y)\right|^p\leq\left(\frac{q}{p}\phi^{\frac{q}{p}-1}(\xi)|D\phi(\xi)||x-y|\right)^p\leq Cr^{-p}|x-y|^p.
$$
Then,
\begin{align*}
&\quad\int_{B_{2r}}\int_{B_{2r}}\left[\left|\phi^\frac{q}{p}(x)-\phi^\frac{q}{p}(y)\right|^pK_{sp}(x,y)
+a(x,y)\frac{|\phi(x)-\phi(y)|^q}{(\|u\|_{L^\infty(\Omega)}+d)^{p-q}}K_{tq}(x,y)\right]\,dx\,dy\\
&\le C\int_{B_{2r}}\int_{B_{2r}}r^{-p}|x-y|^{-n+p(1-s)}+M(\|u\|_{L^\infty(\Omega)}+d)^{q-p}r^{-q}|x-y|^{-n+q(1-t)}\,dx\,dy\\
&\le C\left[r^{n-sp}+M(\|u\|_{L^\infty(\Omega)}+d)^{q-p}r^{n-tq}\right].
\end{align*}
As has been stated above, we eventually get
\begin{align*}
&\quad \int_{B_{r}}\int_{B_{r}}\Bigg[\left|\log\left(\frac{u(x)+d}{u(y)+d}\right)\right|^pK_{sp}(x,y)
+d^{q-p}a(x,y)\left|\log\left(\frac{u(x)+d}{u(y)+d}\right)\right|^qK_{tq}(x,y)\Bigg]\,dx\,dy\\
&\leq Cd^{1-p}r^n\left(R^{-sp}[\mathrm{Tail}(u_-;x_0,R)]^{p-1}+R^{-tq}[\mathrm{Tail}_a(u_-;x_0,R)]^{q-1}\right)\\
&\quad+Cr^n\left[r^{-sp}+M(\|u\|_{L^\infty(\Omega)}+d)^{q-p}r^{-tq}\right],
\end{align*}
where $C$ depends only on $n,p,q,s,t,\Lambda_1$ and $\Lambda_2$.
\end{proof}

A direct consequence of the aforementioned lemma is the following.

\begin{corollary}
\label{cor3-3}
Let $1<p\leq q<\infty$. Suppose that $u\in W^{s,p}(\mathbb{R}^n)\cap L^\infty_{\rm loc}(\mathbb{R}^n)$ is a weak solution to Eq. \eqref{main} satisfying $u\ge0$ in $B_R(x_0)\subset\Omega$. Define
$$
v:=\min\left\{(\log(a+d)-\log(u+d))_+, \log(b)\right\}
$$
with $a,d>0$ and $b>1$. Then for any $B_r:=B_r(x_0)\subset B_{R/2}(x_0)$ there holds that
\begin{align*}
&\quad \fint_{B_r}|v-(v)_{B_r}|^p\,dx\\
&\le Cd^{1-p}r^{sp}\left(R^{-sp}[\mathrm{Tail}(u_-;x_0,R)]^{p-1}+R^{-tq}[\mathrm{Tail}_a(u_-;x_0,R)]^{q-1}\right)\\
&\quad+C\left[1+M(\|u\|_{L^\infty(\Omega)}+d)^{q-p}r^{sp-tq}\right],
\end{align*}
where $C$ depends only on $n,p,q,s,t,\Lambda_1$ and $\Lambda_2$.
\end{corollary}

\begin{proof}
From the fractional Poincar\'{e} type inequality and the condition ($A_2$) on $K_{sp}$, we derive
$$
\fint_{B_r}|v-(v)_{B_r}|^p\,dx\le Cr^{sp-n}\int_{B_r}\int_{B_r}K_{sp}(x,y)|v(x)-v(y)|^p\,dx\,dy
$$
with $C$ depending on $n,p,s,\Lambda_1$. By means of Lemma \ref{lem3-2}, we arrive at
\begin{align*}
&\quad \int_{B_r}\int_{B_r}K_{sp}(x,y)|v(x)-v(y)|^p\,dx\,dy\\
&\leq\int_{B_r}\int_{B_r}K_{sp}(x,y)\left|\log\left(\frac{u(y)+d}{u(x)+d}\right)\right|^p\,dx\,dy\\
&\leq Cd^{1-p}r^n\left(R^{-sp}[\mathrm{Tail}(u_-;x_0,R)]^{p-1}+R^{-tq}[\mathrm{Tail}_a(u_-;x_0,R)]^{q-1}\right)\\
&\quad+Cr^n\left[r^{-sp}+M(\|u\|_{L^\infty(\Omega)}+d)^{q-p}r^{-tq}\right],
\end{align*}
where $C$ depends only on $n,p,q,s,t,\Lambda_1$ and $\Lambda_2$. As a result, we have verified this claim.
\end{proof}

At this point, we shall concentrate on proving the H\"{o}lder continuity of bounded weak solutions to Eq. \eqref{main}. To this end, we will show an iteration lemma that is the key step of the proof. Before starting, let us introduce some notations. For each $j\in\mathbb{N}$, set
$$
r_j:=\sigma^j\frac{r}{2},  \sigma\in\left(0, {1}/{4}\right] \quad\text{and} \quad B_j:=B_{r_j}(x_0),
$$
where $0<r<\frac{R}{2}$ with some $R\leq1$ fulfilling $B_R(x_0)\subset\Omega$. Furthermore, denote
\begin{align*}
\omega(r_0)&:=\|u\|_{L^\infty(\mathbb{R}^n)}+\mathrm{Tail}\left(u;x_0,\frac{r}{2}\right)+\mathrm{Tail}_a\left(u;x_0,\frac{r}{2}\right)\\
&\leq\left(2+M^\frac{1}{q-1}\right)\|u\|_{L^\infty(\mathbb{R}^n)}
\end{align*}
and
$$
\omega(r_j):=\left(\frac{r_j}{r_0}\right)^\alpha\omega(r_0)
$$
for some $\alpha<\min\left\{\frac{sp}{p-1},\frac{tq}{q-1}\right\}$. 

Now we are in a position to infer an oscillation reduction.

\begin{lemma}
\label{lem3-4}
Let the assumption \eqref{2-1} be in force. Suppose that $u\in W^{s,p}(\mathbb{R}^n)\cap L^\infty(\mathbb{R}^n)$ is a weak solution to Eq. \eqref{main}. Then it holds that
$$
\mathrm{osc}_{B_j}u=\sup_{B_j}u-\inf_{B_j}u\le \omega(r_j), \quad\text{for any } j\in\mathbb{N},
$$
where these notations are fixed above.
\end{lemma}

\begin{proof}
We argue by induction. Assume that the claim holds true for $i\le j$. Now we are going to show it also holds for $i=j+1$. We can know that either
\begin{equation}
\label{3-14}
\frac{|2B_{j+1}\cap\{u\ge \inf_{B_j}u+\omega(r_j)/2\}|}{|2B_{j+1}|}\ge\frac{1}{2}
\end{equation}
or
\begin{equation}
\label{3-15}
\frac{|2B_{j+1}\cap\{u<\inf_{B_j}u+\omega(r_j)/2\}|}{|2B_{j+1}|}\ge\frac{1}{2}.
\end{equation}
Define
\begin{equation*}
u_j=\begin{cases}u-\inf_{B_j}u, &\text{\textmd{if \eqref{3-14} holds}},\\[2mm]
\omega(r_j)-(u-\inf_{B_j}u), &\text{\textmd{if \eqref{3-15} holds}}.
\end{cases}
\end{equation*}
Obviously, $u_j\ge0$ in $B_j$ and
\begin{equation}
\label{3-16}
\frac{|2B_{j+1}\cap\{u_j\geq\omega(r_j)/2\}|}{|2B_{j+1}|}\ge\frac{1}{2}.
\end{equation}
Moreover, $u_j$ is a weak solution such that
\begin{equation}
\label{3-17}
\sup_{B_i}|u_j|\le 2\omega(r_i) \quad\text{for any } i\in\{0,1,2,\cdots,j\}.
\end{equation}
We next introduce an auxiliary function
$$
v:=\min\left\{\left[\log\left(\frac{\omega(r_j)/2+d}{u_j+d}\right)\right]_+,k\right\} \quad\text{with } k>0.
$$
Applying Corollary \ref{cor3-3}, we obtain
\begin{align}
\label{3-18}
&\quad\fint_{2B_{j+1}}|v-(v)_{2B_{{j+1}}}|^p\,dx \nonumber\\
&\leq Cd^{1-p}\left[\left(\frac{r_{j+1}}{r_j}\right)^{sp}[\mathrm{Tail}(u_j;x_0,r_j)]^{p-1}+\frac{r_{j+1}^{sp}}{r_j^{tq}}[\mathrm{Tail}_a(u_j;x_0,r_j)]^{q-1}\right] \nonumber\\
&\quad+C\left[1+M(\|u\|_{L^\infty(\mathbb{R}^n)}+d)^{q-p}r_{j+1}^{sp-tq}\right].
\end{align}
After calculation, we get
\begin{equation}
\label{3-19}
[\mathrm{Tail}(u_j;x_0,r_j)]^{p-1}\leq C\sigma^{-\alpha(p-1)}\omega(r_j)^{p-1}
\end{equation}
with $\alpha<\frac{sp}{p-1}$, where $C$ depends on $n,p,s,\alpha$. For the details, one can refer to \cite[page 1295]{DKP16}. We now estimate
\begin{align*}
&\quad[\mathrm{Tail}_a(u_j;x_0,r_j)]^{q-1}\\
&=r_j^{tq}\sup_{x\in B_{r_j}(x_0)}\int_{\mathbb{R}^n\setminus B_{r_j}(x_0)}a(x,y)\frac{|u_j(y)|^{q-1}}{|y-x_0|^{n+tq}}\,dy\\
&\le r^{tq}_j\sum^j_{i=1}\sup_{x\in B_{i-1}}\int_{B_{i-1}\setminus B_i}a(x,y)|u_j(y)|^{q-1}|y-x_0|^{-n-tq}\,dy\\
&\quad+r^{tq}_j\sup_{x\in B_0}\int_{\mathbb{R}^n\setminus B_0}a(x,y)|u_j(y)|^{q-1}|y-x_0|^{-n-tq}\,dy\\
&\le r^{tq}_j\sum^j_{i=1}M\left(\sup_{B_{i-1}}|u_j(y)|\right)^{q-1}\int_{B_{i-1}\setminus B_i}|y-x_0|^{-n-tq}\,dy\\
&\quad+r^{tq}_j\sup_{x\in B_0}\int_{\mathbb{R}^n\setminus B_0}a(x,y)|u_j(y)|^{q-1}|y-x_0|^{-n-tq}\,dy\\
&\leq Mr^{tq}_j\sum^j_{i=1}(2\omega(r_{i-1}))^{q-1}r_i^{-tq}+r^{tq}_jM\omega(r_0)^{q-1}r_1^{-tq}\\
&\leq CM\sum^j_{i=1}\left(\frac{r_j}{r_i}\right)^{tq}\omega(r_{i-1})^{q-1}
\end{align*}
with $C$ depending on $q$ only, where we have used the inequality \eqref{3-17} and
\begin{align*}
&\quad r^{tq}_j\sup_{x\in B_0}\int_{\mathbb{R}^n\setminus B_0}a(x,y)|u_j(y)|^{q-1}|y-x_0|^{-n-tq}\,dy\\
&\le r^{tq}_j\sup_{x\in B_0}\int_{\mathbb{R}^n\setminus B_0}a(x,y)\left(|u(y)|^{q-1}+\omega(r_0)^{q-1}+\sup_{B_0}|u|^{q-1}\right)|y-x_0|^{-n-tq}\,dy\\
&\le r^{tq}_j\left(M\omega(r_0)^{q-1}r_0^{-tq}+\sup_{x\in B_0}\int_{\mathbb{R}^n\setminus B_0}a(x,y)|u(y)|^{q-1}|y-x_0|^{-n-tq}\,dy\right)\\
&\le M\left(\frac{r_j}{r_1}\right)^{tq}\omega(r_0)^{q-1}.
\end{align*}
We can readily evaluate
\begin{align}
\label{3-20}
\sum^j_{i=1}\left(\frac{r_j}{r_i}\right)^{tq}\omega(r_{i-1})^{q-1}&\leq \frac{4^{tq-\alpha(q-1)}}{(tq-\alpha(q-1))\log4}\sigma^{-\alpha(q-1)}\omega(r_j)^{q-1} \nonumber\\
&=:C\sigma^{-\alpha(q-1)}\omega(r_j)^{q-1},
\end{align}
where we utilized the fact that $\alpha<\frac{tq}{q-1}$, and $C$ depends on $n,q,t,\alpha$. Combining \eqref{3-18}, \eqref{3-19} with \eqref{3-20}, we have
\begin{align*}
&\quad\fint_{2B_{j+1}}|v-(v)_{2B_{{j+1}}}|^p\,dx\\
&\leq Cd^{1-p}\left[\sigma^{sp-\alpha(p-1)}\omega(r_j)^{p-1}+Mr_{j+1}^{sp-tq}\sigma^{tq-\alpha(q-1)}\omega(r_j)^{q-1}\right]\\
&\quad+C\left[1+Mr_{j+1}^{sp-tq}(\|u\|_{L^\infty(\mathbb{R}^n)}+d)^{q-p}\right],
\end{align*}
where $C$ depends on $n,p,q,s,t,\alpha$.

In the sequel, selecting
$$
d:=\sigma^{\frac{tq}{q-1}-\alpha}\omega(r_j),
$$
by $\omega(r_j)=\left(\frac{r_j}{r_0}\right)^\alpha\omega(r_0)$, we obtain
\begin{align}
\label{3-21}
&\quad\fint_{2B_{j+1}}|v-(v)_{2B_{{j+1}}}|^p\,dx \nonumber\\
&\leq C\left[\sigma^{\left({\frac{tq}{q-1}-\alpha}\right)(1-p)+\left({\frac{sp}{p-1}-\alpha}\right)(p-1)}
+\sigma^{\left({\frac{tq}{q-1}-\alpha}\right)(1-p)+\left({\frac{tq}{q-1}-\alpha}\right)(q-1)}\omega(r_j)^{q-p}\right] \nonumber\\
&\quad+C\left[1+M\left(\|u\|_{L^\infty(\mathbb{R}^n)}+\sigma^{\frac{tq}{q-1}-\alpha}\omega(r_j)\right)^{q-p}\right] \nonumber\\
&\leq C(1+\omega(r_j)^{q-p})+C\left(1+M\|u\|_{L^\infty(\mathbb{R}^n)}^{q-p}+M\omega(r_j)^{q-p}\right) \nonumber\\
&\leq C\left(1+M\|u\|_{L^\infty(\mathbb{R}^n)}^{q-p}+M\omega(r_j)^{q-p}\right) \nonumber\\
&\leq C\left(1+M\|u\|_{L^\infty(\mathbb{R}^n)}^{q-p}+M\omega(r_0)^{q-p}\right).
\end{align}
We have an estimate on $k$,
$$
k\le 2(k-(v)_{\tilde{B}})
$$
with $\tilde{B}:=2B_{j+1}$. We refer to \cite[page 1296]{DKP16} for the details. Furthermore,
\begin{align*}
\frac{|\tilde{B}\cap\{v=k\}|}{|\tilde{B}|}
&\le \frac{2}{|\tilde{B}|}\int_{\tilde{B}\cap\{v=k\}}(k-(v)_{\tilde{B}})\,dx\\
&\le \frac{2}{|\tilde{B}|}\int_{\tilde{B}}|v-(v)_{\tilde{B}}|\,dx\\
&\le C\left(1+M\|u\|_{L^\infty(\mathbb{R}^n)}^{q-p}+M\omega(r_0)^{q-p}\right),
\end{align*}
where $C$ depends on $n,p,q,s,t$ and $\alpha$. In what follows, we denote in short
$$
1+M\|u\|_{L^\infty(\mathbb{R}^n)}^{q-p}+M\omega(r_0)^{q-p}\leq1+\left[M+\left(2+M^\frac{1}{q-1}\right)^{q-p}\right]\|u\|_{L^\infty(\mathbb{R}^n)}^{q-p}=:A.
$$
By taking
$$
k=\log\left(\frac{\omega(r_j)/2+\varepsilon\omega(r_j)}{3\varepsilon\omega(r_j)}\right)=\log\left(\frac{1/2+\varepsilon}{3\varepsilon}\right)\approx\log\frac{1}{\varepsilon}
$$
with $\varepsilon:=\sigma^{\frac{tq}{q-1}-\alpha}$, it yields that
\begin{equation}
\label{3-22}
\frac{|\tilde{B}\cap\{u_j\leq2\varepsilon\omega(r_j)\}|}{|\tilde{B}|}\le \frac{CA}{k}\le \frac{C_{\rm log}A}{\log\frac{1}{\sigma}}
\end{equation}
with $C_{\rm log}$ depending on $n,p,q,s,t,\Lambda_1,\Lambda_2$ and $\alpha$.

We next proceed with an suitable iteration procedure. First, for each $i=0,1,\cdots$, we define
$$
\rho_i=r_{j+1}+2^{-i}r_{j+1}, \quad \overline{\rho_i}=\frac{\rho_i+\rho_{i+1}}{2}, \quad B^i=B_{\rho_i}, \quad \overline{B}^i=B_{\overline{\rho_i}}
$$
and the corresponding cut-off functions
$$
\phi_i\in C^\infty_0(\overline{B}^i), \quad 0\le \phi_i\le 1, \quad \phi_i\equiv 1 \text{ in } B^{i+1}, \quad |D\phi_i|\le C\rho^{-1}_i.
$$
In addition, set
$$
k_i=(1+2^{-i})\varepsilon\omega(r_j), \quad w_i=(k_i-u_j)_+
$$
and
$$
A_i=\frac{|B^i\cap \{u_j\leq k_i\}|}{|B^i|}=\frac{|B^i\cap \{w_j>0\}|}{|B^i|}.
$$
We employ the Caccioppoli inequality in Lemma \ref{lem3-1} to derive
\begin{align}
\label{3-23}
&\quad\int_{B^i}\int_{B^i}\left|w_i(x)\phi_i^\frac{q}{p}(x)-w_i(y)\phi_i^\frac{q}{p}(x)\right|^pK_{sp}(x,y)\,dx\,dy \nonumber\\
&\le C\int_{B^i}\int_{B^i}\bigg[(\max\{w_i(x),w_i(y)\})^p\left|\phi_i^\frac{q}{p}(x)-\phi_i^\frac{q}{p}(y)\right|^pK_{sp}(x,y) \nonumber\\
&\qquad\qquad+a(x,y)(\max\{w_i(x),w_i(y)\})^q|\phi_i(x)-\phi_i(y)|^qK_{tq}(x,y)\bigg]\,dx\,dy  \nonumber\\
&\quad +C\int_{B^i}w_i(x)\phi_i^q(x)\,dx\left(\sup_{x\in \overline{B}^i}\int_{\mathbb{R}^n\setminus B^i}w_i^{p-1}(y)K_{sp}(x,y)+a(x,y)w_i^{q-1}(y)K_{tq}(x,y)\,dy\right) \nonumber\\
&=:J_1+J_2.
\end{align}
We first evaluate
\begin{align}
\label{3-24}
&\quad A^\frac{p}{p^*}_{i+1}(k_i-k_{i+1})^p \nonumber\\
&=\frac{1}{|B^{i+1}|^\frac{p}{p^*}}\left(\int_{B^{i+1}\cap\{u_j\le k_{i+1}\}}(k_i-k_{i+1})\phi^{\frac{q}{p}p^*}_i\,dx\right)^\frac{p}{p^*} \nonumber\\
&\leq \frac{1}{|B^{i+1}|^\frac{p}{p^*}}\left(\int_{B^i}w_i^{p^*}\phi^{\frac{q}{p}p^*}_i\,dx\right)^\frac{p}{p^*} \nonumber\\
&\leq Cr^{sp-n}_{j+1}\int_{B^i}\int_{B^i}\left|w_i(x)\phi_i^\frac{q}{p}(x)-w_i(y)\phi_i^\frac{q}{p}(y)\right|^pK_{sp}(x,y)\,dx\,dy.
\end{align}
Next we consider the integral $J_1$. Notice that
$$
\left|\phi_i^\frac{q}{p}(x)-\phi_i^\frac{q}{p}(y)\right|^p\le C(n,p,q)2^{ip}r^{-p}_{j+1}|x-y|^p
$$
and
$$
|\phi_i(x)-\phi_i(y)|^q\le C(n,q)2^{iq}r^{-q}_{j+1}|x-y|^q.
$$
Hence, using the assumption ($A_2$), we get
\begin{align}
\label{3-25}
r^{sp}_{j+1}J_1&\leq C\int_{B^i\cap \{u_j\le k_i\}}\int_{B^i}2^{ip}r^{sp-p}_{j+1}k_i^p|x-y|^{p-n-sp}+2^{iq}Mr^{sp-q}_{j+1}k_i^q|x-y|^{q-n-tq}\,dy\,dx \nonumber\\
&\leq C2^{ip}k_i^p|B^i\cap\{u_j\leq k_i\}|+CM2^{iq}k_i^qr^{sp-tq}_{j+1}|B^i\cap\{u_j\leq k_i\}| \nonumber\\
&\leq C|B^i\cap\{u_j\leq k_i\}|\left[2^{ip}(\varepsilon\omega(r_j))^p+MR^{sp-tq}2^{iq}(\varepsilon\omega(r_j))^q\right],
\end{align}
where $C$ depends on $n,p,q,s,t,\Lambda_1$ and $\Lambda_2$. We proceed by evaluating $J_2$. It is easy to get
\begin{equation}
\label{3-26}
\int_{B^i}w_i(x)\phi_i^q(x)\,dx\le C\varepsilon\omega(r_j)|B^i\cap\{u_j\leq k_i\}|.
\end{equation}
In order to handle the third integral on the right-hand side of \eqref{3-23}, we first arrive at
\begin{align}
\label{3-27}
r^{sp}_{j+1}\left(\sup_{x\in \overline{B}^i}\int_{\mathbb{R}^n\setminus B^i}w_i^{p-1}(y)K_{sp}(x,y)\,dy\right)&\leq C2^{i(n+sp)}[\mathrm{Tail}(w_i;x_0,r_{j+1})]^{p-1} \nonumber\\
&\leq C2^{i(n+sp)}(\varepsilon\omega(r_j))^{p-1}
\end{align}
and in view of the condition ($A_2$), $B_{j+1}\subset B^i\subset B_j$ and \eqref{3-20}, we get
\begin{align}
\label{3-28}
&\quad r^{tq}_{j+1}\left(\sup_{x\in \overline{B}^i}\int_{\mathbb{R}^n\setminus B^i}a(x,y)w_i^{q-1}(y)K_{tq}(x,y)\,dy\right) \nonumber\\
&\leq Cr^{tq}_{j+1}\left(\sup_{x\in \overline{B}^i}\int_{\mathbb{R}^n\setminus B^i}a(x,y)\frac{2^{i(n+tq)}w^{q-1}_i(y)}{|y-x_0|^{n+tq}}\,dy\right) \nonumber\\
&\leq C2^{i(n+tq)}r^{tq}_{j+1}\sup_{x\in \overline{B}^i}\left(\int_{\mathbb{R}^n\setminus B_j}+\int_{B_j\setminus B_{j+1}}a(x,y)\frac{w^{q-1}_i(y)}{|y-x_0|^{n+tq}}\,dy\right) \nonumber\\
&\leq C2^{i(n+tq)}r^{tq}_{j+1}\left(\sup_{x\in B_j}\int_{\mathbb{R}^n\setminus B_j}a(x,y)\frac{w^{q-1}_i(y)}{|y-x_0|^{n+tq}}\,dy+M\int_{B_j\setminus B_{j+1}}\frac{(2\varepsilon\omega(r_j))^{q-1}}{|y-x_0|^{n+tq}}\,dy\right) \nonumber\\
&\leq C2^{i(n+tq)}r^{tq}_{j+1}\left(M(\varepsilon\omega(r_j))^{q-1}r^{-tq}_{j+1}+\sup_{x\in B_j}\int_{\mathbb{R}^n\setminus B_j}a(x,y)\frac{|u_j|^{q-1}+(\varepsilon\omega(r_j))^{q-1}}{|y-x_0|^{n+tq}}\,dy\right) \nonumber\\
&\leq C2^{i(n+tq)}\left[M(\varepsilon\omega(r_j))^{q-1}+\left(\frac{r_{j+1}}{r_j}\right)^{tq}(\mathrm{Tail_a}(u_j;x_0,r_j))^{q-1}\right] \nonumber\\
&\leq C2^{i(n+tq)}\left(M(\varepsilon\omega(r_j))^{q-1}+\sigma^{tq}M\sigma^{-\alpha(q-1)}(\omega(r_j))^{q-1}\right) \nonumber\\
&=CM2^{i(n+tq)}\left(1+\frac{\sigma^{tq-\alpha(q-1)}}{\varepsilon^{q-1}}\right)(\varepsilon\omega(r_j))^{q-1} \nonumber\\
&=CM2^{i(n+tq)}(\varepsilon\omega(r_j))^{q-1}.
\end{align}
Putting together \eqref{3-23}--\eqref{3-28}, we arrive at
\begin{align*}
&\quad A^\frac{p}{p^*}_{i+1}(k_i-k_{i+1})^p\\
&\leq r^{-n}_{j+1}\Big[C|B^i\cap\{u_j\leq k_i\}|\left(2^{ip}(\varepsilon\omega(r_j))^p+M2^{iq}(\varepsilon\omega(r_j))^q\right)\\
&\qquad\quad+C|B^i\cap\{u_j\leq k_i\}|\left(2^{i(n+sp)}(\varepsilon\omega(r_j))^p+M2^{i(n+tq)}(\varepsilon\omega(r_j))^q\right)\Big]\\
&\leq CA_i\Big[(1+M)2^{iq}((\varepsilon\omega(r_j))^p+(\varepsilon\omega(r_j))^q)\\
&\qquad\quad+(1+M)2^{i(n+sp)}((\varepsilon\omega(r_j))^p+(\varepsilon\omega(r_j))^q)\Big]\\
&\leq C(1+M)2^{i(n+q+sp)}(\varepsilon\omega(r_j))^p[1+(\varepsilon\omega(r_j))^{q-p}]A_i,
\end{align*}
which leads to
$$
A^\frac{p}{p^*}_{i+1}\leq C(1+M)2^{i(n+q+sp)}[1+\omega(r_0)^{q-p}]A_i
$$
and further
$$
A_{i+1}\leq C2^{i(n+q+sp)\frac{p^*}{p}}A^\frac{p^*}{p}A_i^{1+\beta}
$$
with $\beta=\frac{sp}{n-sp}$, where $C$ depends on $n,p,q,s,t,\Lambda_1,\Lambda_2,\alpha$ and $M$. Now if we obtain the following estimate on $A_0$,
\begin{equation}
\label{3-29}
A_0=\frac{|\tilde{B}\cap\{u_j\leq 2\varepsilon\omega(r_j)\}|}{|\tilde{B}|}\leq C^{-\frac{1}{\beta}}A^{-\frac{p^*}{\beta p}}2^{-\frac{(n+q+sp)p^*}{p\beta^2}}:=\mu,
\end{equation}
then we conclude that
$$
A_i\rightarrow0 \quad\text{as } i\rightarrow\infty.
$$
From \eqref{3-22}, it follows that
$$
\frac{C_{\rm log}A}{\log\frac{1}{\sigma}}\leq \mu \Rightarrow \sigma\le \mathrm{exp}\left(-\frac{C_{\rm log}A}{\mu}\right).
$$
That is,
$$
\sigma\le \mathrm{exp}\left(-CA^{1+\frac{p^*}{\beta p}}\right).
$$
We now pick
$$
\sigma=\min\left\{\frac{1}{4},\mathrm{exp}\left(-CA^{1+\frac{p^*}{\beta p}}\right)\right\},
$$
which ensures that the condition \eqref{3-29} holds true. In other words, we have proved that
$$
\mathrm{osc}_{B_{j+1}}u\leq (1-\varepsilon)\omega(r_j)=(1-\varepsilon)\sigma^{-\alpha}\omega(r_{j+1}).
$$
Finally, we choose $\alpha\in\left(0,\frac{tq}{q-1}\right)$ small such that
$$
\sigma^\alpha\geq 1-\varepsilon=1-\sigma^{\frac{tq}{q-1}-\alpha},
$$
which implies that
$$
\mathrm{osc}_{B_{j+1}}u\leq\omega(r_{j+1}).
$$
Therefore, we can find that $\alpha$ depends on $n,p,q,s,t,\Lambda_1,\Lambda_2$ and $M$. The proof is completed now.
\end{proof}

\section{Weak solutions are viscosity solutions}
\label{sec-4}

Throughout this part, we always assume that the conditions \eqref{2-1}, ($A_1$)--($A_4$), ($H_1'$) and ($H_2$)--($H_4$) hold true. The aim of this section is to verify that bounded weak solutions are viscosity solutions to Eq. \eqref{main}. One of indispensable ingredients of the proof is the comparison principle for weak solutions, which is stated as follows.

\begin{proposition}[Comparison principle]
\label{pro4-1}
Let $\Omega'\subset\subset\Omega$. Assume that $u,v$ are a weak supersolution and a weak subsolution to Eq. \eqref{main} in $\Omega$, respectively. If $u\geq v$ almost everywhere in $\mathbb{R}^n\setminus\Omega'$, then it holds that
$$
u\geq v \quad\text{almost everywhere in } \Omega'.
$$
\end{proposition}

The proof of this proposition is similar to that of Lemma 6 in \cite{KKP17}, so we omit it here.
We now provide the following trivial but very important lemma.

\begin{lemma}
\label{lem4-2}
Let $l$ be an affine function and $r\in(0,+\infty)$. Then there holds that
$$
\int_{B_r(x)\setminus B_\varepsilon(x)}a(x,y)|l(x)-l(y)|^{q-2}(l(x)-l(y))K_{tq}(x,y)\,dy=0
$$
for any $\varepsilon\in(0,r)$.
\end{lemma}

\begin{proof}
Let $l(x)=a+b\cdot x$. Then by translation invariance and symmetry properties of $a(x,y)$ and $K_{tq}(x,y)$, we have
\begin{align*}
&\quad \int_{B_r(x)\setminus B_\varepsilon(x)}a(x,y)|b\cdot(x-y)|^{q-2}b\cdot(x-y)K_{tq}(x,y)\,dy\\
&=-\int_{B_r\setminus B_\varepsilon}a(x,x+z)|b\cdot z|^{q-2}b\cdot zK_{tq}(x,x+z)\,dz\\
&=-\int_{B_r\setminus B_\varepsilon}a(0,z)|b\cdot z|^{q-2}b\cdot zK_{tq}(0,z)\,dz\\
&=0.
\end{align*}
We get the desired result.
\end{proof}

Next, we are ready to demonstrate that the principle value defining the nonlocal double phase operator $\mathcal{L}$ is well-defined provided that the involved functions are regular enough. To this end, we have to establish the forthcoming two uniform estimates on small balls. For simplicity, we set $g_k(t):=|t|^{k-2}t$ ($k\in\{p,q\}$) in the sequel.

\begin{lemma}
\label{lem4-3}
Let $B_\varepsilon(x)\subset E\subset\subset\Omega$ and $u\in C^2(E)$. If $p>\frac{2}{2-s}$ or $E\subset\subset\{d_u>0\}$, then there holds that
\begin{align*}
\Bigg|\mathrm{P.V.}\int_{B_\varepsilon(x)}&\Big[|u(x)-u(y)|^{p-2}(u(x)-u(y))K_{sp}(x,y)\\
&+a(x,y)|u(x)-u(y)|^{q-2}(u(x)-u(y))K_{tq}(x,y)\Big]\,dy\Bigg|\leq C(\varepsilon)
\end{align*}
with $C(\varepsilon)$ independent of $x$ and $C(\varepsilon)\rightarrow0$ as $\varepsilon\rightarrow0$.
\end{lemma}

\begin{proof}
Notice an obvious fact that if $p>\frac{2}{2-s}$, then $\frac{2}{2-t}\leq\frac{2}{2-s}<p\leq q$. If $|Du(x)|=0$ and $p>\frac{2}{2-s}$, we can readily verify this claim by
$$
|u(x)-u(y)|\leq C|x-y|^2
$$
for some constant $C$ (depending only on $\|u\|_{C^2(E)}$). Next we focus on the scenario $Du(x)\neq0$. Suppose that $l(y):=u(x)+Du(x)\cdot(y-x)$ is an affine part of $u$ near $x$. It follows from Proposition \ref{pro2-2} and Lemma \ref{lem4-2} that
\begin{align*}
&\quad\Bigg|\int_{B_\varepsilon(x)}\Big[|u(x)-u(y)|^{p-2}(u(x)-u(y))K_{sp}(x,y)\\
&\qquad\qquad+a(x,y)|u(x)-u(y)|^{q-2}(u(x)-u(y))K_{tq}(x,y)\Big]\,dy\Bigg|\\
&=\Bigg|\int_{B_\varepsilon(x)}\Big[\big(g_p(u(x)-u(y))-g_p(l(x)-l(y))\big)K_{sp}(x,y)+g_p(l(x)-l(y))K_{sp}(x,y)\\
&\qquad\qquad+a(x,y)\big[\big(g_q(u(x)-u(y))-g_q(l(x)-l(y))\big)+g_q(l(x)-l(y))\big]K_{tq}(x,y)\Big]\,dy\Bigg|\\
&\leq\int_{B_\varepsilon(x)}\Big[\big|g_p(u(x)-u(y))-g_p(l(x)-l(y))\big|K_{sp}(x,y)\\
&\qquad\qquad+a(x,y)\big|g_q(u(x)-u(y))-g_q(l(x)-l(y))\big|K_{tq}(x,y)\Big]\,dy\\
&\leq C\int_{B_\varepsilon(x)}\big(|Du(x)\cdot(y-x)|+|u(y)-l(y)|\big)^{p-2}|u(y)-l(y)|K_{sp}(x,y)\,dy\\
&\quad+CM\int_{B_\varepsilon(x)}\big(|Du(x)\cdot(y-x)|+|u(y)-l(y)|\big)^{q-2}|u(y)-l(y)|K_{tq}(x,y)\,dy\\
&=:I_1+I_2.
\end{align*}
For the integral $I_1$, as the proof of Lemma 3.6 \cite{KKL19}, we have
\begin{equation*}
I_1\leq\begin{cases}C\tau\sup_E|Du|^{p-2}\varepsilon^{p(1-s)}+C\tau^{p-1}\varepsilon^{p-2+p(1-s)}  &\text{\textmd{for }} p\geq2,\\[2mm]
C\tau^{p-1}\varepsilon^{p-2+p(1-s)}  &\text{\textmd{for }} \frac{2}{2-s}<p<2,
\end{cases}
\end{equation*}
where $\tau:=\sup_E|D^2u|$. Similarly, for $I_2$ we obtain
\begin{equation}
\label{4-1}
I_2\leq\begin{cases}CM\tau\sup_E|Du|^{q-2}\varepsilon^{q(1-t)}+CM\tau^{q-1}\varepsilon^{q-2+q(1-t)}  &\text{\textmd{for }} q\geq2,\\[2mm]
CM\tau^{q-1}\varepsilon^{q-2+q(1-t)}  &\text{\textmd{for }} \frac{2}{2-t}<q<2.
\end{cases}
\end{equation}
Hence, combining these two inequalities, we get
\begin{align*}
&\quad\Bigg|\int_{B_\varepsilon(x)}\left(g_p(u(x)-u(y))K_{sp}(x,y)+a(x,y)g_q(u(x)-u(y))K_{tq}(x,y)\right)\,dy\Bigg|\\
&\leq\begin{cases}C\left(\varepsilon^{p(1-s)}+\varepsilon^{q(1-t)}+\varepsilon^{p-2+p(1-s)}+\varepsilon^{q-2+q(1-t)}\right) &\text{\textmd{for }} p\geq 2,\\[2mm]
C\left(\varepsilon^{q(1-t)}+\varepsilon^{p-2+p(1-s)}+\varepsilon^{q-2+q(1-t)}\right) &\text{\textmd{for }} \frac{2}{2-s}<p<2\leq q,\\[2mm]
C\left(\varepsilon^{p-2+p(1-s)}+\varepsilon^{q-2+q(1-t)}\right)  &\text{\textmd{for }} \frac{2}{2-s}<p\leq q<2.
\end{cases}
\end{align*}

If $1<p\leq\frac{2}{2-s}$ and $E\subset\subset\{d_u(x)>0\}$, we arrive at
$$
I_1\leq C\tau\sup_E|Du|^{p-2}\varepsilon^{p(1-s)}.
$$
Here we used the fact that $\inf_E|Du|>0$. Moreover, we estimate
$$
I_2\leq CM\tau\sup_E|Du|^{q-2}\varepsilon^{q(1-t)},
$$
provided that $1<q\leq \frac{2}{2-t}$ ($1<p\leq q\leq\frac{2}{2-t}\leq\frac{2}{2-s}$). Therefore, in the case $1<p\leq \frac{2}{2-s}$,
\begin{align*}
&\quad\Bigg|\int_{B_\varepsilon(x)}\left(g_p(u(x)-u(y))K_{sp}(x,y)+a(x,y)g_q(u(x)-u(y))K_{tq}(x,y)\right)\,dy\Bigg|\\
&\leq\begin{cases}C\left(\varepsilon^{p(1-s)}+\varepsilon^{q(1-t)}\right) &\text{\textmd{for }} 1<q\leq \frac{2}{2-t},\\[2mm]
C\left(\varepsilon^{p(1-s)}+\varepsilon^{q-2+q(1-t)}\right) &\text{\textmd{for }} \frac{2}{2-t}<q<2,\\[2mm]
C\left(\varepsilon^{p(1-s)}+\varepsilon^{q(1-t)}+\varepsilon^{q-2+q(1-t)}\right)  &\text{\textmd{for }}  q\geq 2.
\end{cases}
\end{align*}
Now in all cases, this claim is proved.
\end{proof}


\begin{lemma}
\label{lem4-4}
Let $1<p\leq \frac{2}{2-s}$, $E\subset\subset\Omega$ and $u\in C^2_\beta(E)$ with $\beta>\frac{sp}{p-1}$. Assume that $B_\varepsilon(x)\subset E$ and $x$ is such that $d_u(x)<\varepsilon<1$. Then it holds that
\begin{align*}
\Bigg|\mathrm{P.V.}\int_{B_\varepsilon(x)}&\Big[|u(x)-u(y)|^{p-2}(u(x)-u(y))K_{sp}(x,y)\\
&+a(x,y)|u(x)-u(y)|^{q-2}(u(x)-u(y))K_{tq}(x,y)\Big]\,dy\Bigg|\leq C(\varepsilon)
\end{align*}
with $C(\varepsilon)$ independent of $x$ and $C(\varepsilon)\rightarrow0$ as $\varepsilon\rightarrow0$.
\end{lemma}

\begin{proof}
If $Du(x)=0$, this conclusion can be deduced easily because $u\in C^2_\beta(E)$ implies that
$$
|u(x)-u(y)|\leq C|x-y|^\beta
$$
for some constant $C>0$. Next we concentrate on the scenario that $Du(x)\neq0$. Suppose that $l(y):=u(x)+Du(x)\cdot(y-x)$ is an affine part of $u$ near $x$. We derive through Lemma \ref{lem4-2} and Proposition \ref{pro2-2} that
\begin{align}
\label{4-2}
&\quad\Bigg|\int_{B_\varepsilon(x)}\Big[|u(x)-u(y)|^{p-2}(u(x)-u(y))K_{sp}(x,y)   \nonumber\\
&\qquad\qquad+a(x,y)|u(x)-u(y)|^{q-2}(u(x)-u(y))K_{tq}(x,y)\Big]\,dy\Bigg|   \nonumber\\
&\leq\int_{B_\varepsilon(x)}\Big[\big|g_p(u(x)-u(y))-g_p(l(x)-l(y))\big|K_{sp}(x,y)    \nonumber\\
&\qquad\qquad+a(x,y)\big|g_q(u(x)-u(y))-g_q(l(x)-l(y))\big|K_{tq}(x,y)\Big]\,dy    \nonumber\\
&\leq C\int_{B_\varepsilon(x)}\big(|Du(x)\cdot(y-x)|+|u(y)-l(y)|\big)^{p-2}|u(y)-l(y)|K_{sp}(x,y)\,dy  \nonumber\\
&\quad+CM\int_{B_\varepsilon(x)}\big(|Du(x)\cdot(y-x)|+|u(y)-l(y)|\big)^{q-2}|u(y)-l(y)|K_{tq}(x,y)\,dy  \nonumber\\
&=:I_1+I_2.
\end{align}
Analogously to the proof of Lemma 3.7 in \cite{KKL19}, we know that
\begin{equation}
\label{4-3}
I_1\leq C\varepsilon^{\beta(p-1)-sp}.
\end{equation}
When $1<q\leq \frac{2}{2-t}$ ($1<p\leq q\leq\frac{2}{2-t}\leq\frac{2}{2-s}$), then it yields that
\begin{equation}
\label{4-4}
I_2\leq CM\varepsilon^{\beta(q-1)-tq},
\end{equation}
in a similar way to evaluating $I_1$. Here we utilized the fact that $\beta>\frac{sp}{p-1}(\geq\frac{tq}{q-1})$ in the estimates on $I_1,I_2$. Thereby, putting together \eqref{4-1}--\eqref{4-4}, we obtain
\begin{align*}
&\quad\Bigg|\int_{B_\varepsilon(x)}\left(g_p(u(x)-u(y))K_{sp}(x,y)+a(x,y)g_q(u(x)-u(y))K_{tq}(x,y)\right)\,dy\Bigg|\\
&\leq\begin{cases}C\left(\varepsilon^{\beta(p-1)-sp}+\varepsilon^{\beta(q-1)-tq}\right) &\text{\textmd{for }} 1<q\leq \frac{2}{2-t},\\[2mm]
C\left(\varepsilon^{\beta(p-1)-sp}+\varepsilon^{q-2+q(1-t)}\right) &\text{\textmd{for }} \frac{2}{2-t}<q<2,\\[2mm]
C\left(\varepsilon^{\beta(p-1)-sp}+\varepsilon^{q(1-t)}+\varepsilon^{q-2+q(1-t)}\right)  &\text{\textmd{for }} q\geq 2.
\end{cases}
\end{align*}
We now complete the proof.
\end{proof}

Next, we shall prove the continuity property for the nonlocal double phase operator $\mathcal{L}$.

\begin{lemma}
\label{lem4-5}
Let $B_r(x_0)\subset\Omega$ and $\psi\in C^2(B_r(x_0))\cap L^{p-1}_{sp}(\mathbb{R}^n)\cap L^{q-1}_{a,tq}(\mathbb{R}^n)$. When $1<p\leq \frac{2}{2-s}$ and $D\psi(x_0)=0$, we suppose that $\psi\in C^2_\beta(B_r(x_0))$ with $\beta>\frac{sp}{p-1}$. Then $\mathcal{L}\psi$ is continuous in $B_r(x_0)$.
\end{lemma}

\begin{proof}
Let $x\in B_r(x_0)$ and $\varepsilon>0$. Our goal is to prove that
$$
|\mathcal{L}\psi(y)-\mathcal{L}\psi(x)|<\varepsilon
$$
as long as $|y-x|$ is small enough. When $D\psi(x)\neq0$, then by the $C^2$-regularity of $\psi$ we get $D\psi(y)\neq0$ with $|x-y|\leq \delta$ for some $\delta>0$. Via Lemma \ref{lem4-3}, it yields that there is a small enough constant $\rho>0$ such that
\begin{align}
\label{4-5}
\Bigg|\mathrm{P.V.}\int_{B_\rho(y)}&\Big[|\psi(y)-\psi(z)|^{p-2}(\psi(y)-\psi(z))K_{sp}(y,z)  \nonumber\\
&+a(y,z)|\psi(y)-\psi(z)|^{q-2}(\psi(y)-\psi(z))K_{tq}(y,z)\Big]\,dz\Bigg|<\frac{\varepsilon}{3}
\end{align}
with $|y-x|<\delta$. When $p>\frac{2}{2-s}$, we have \eqref{4-5} by using Lemma \ref{lem4-3} again (regardless of the value of $D\psi(x)$). In turn, when $1<p\leq \frac{2}{2-s}$ and $D\psi(x)=0$, then we can see $d_\psi(y)<\rho$ whenever $|x-y|<\rho$ and thus we also get \eqref{4-5} according to Lemma \ref{lem4-4}. In addition, we arrive at
\begin{align}
\label{4-6}
\Bigg|\mathrm{P.V.}\int_{B_\rho(x)}&\Big[|\psi(x)-\psi(z)|^{p-2}(\psi(x)-\psi(z))K_{sp}(x,z)  \nonumber\\
&+a(x,z)|\psi(x)-\psi(z)|^{q-2}(\psi(x)-\psi(z))K_{tq}(x,z)\Big]\,dz\Bigg|<\frac{\varepsilon}{3}
\end{align}
in the case $1<p\leq q<\infty$.

We next consider the nonlocal contribution. We first could readily find that
\begin{align*}
&\quad (1-\chi_{B_\rho(y)}(z))\Big[|\psi(y)-\psi(z)|^{p-2}(\psi(y)-\psi(z))K_{sp}(y,z)\\
&\qquad\qquad\qquad\quad+a(y,z)|\psi(y)-\psi(z)|^{q-2}(\psi(y)-\psi(z))K_{tq}(y,z)\Big]\\
&\stackrel{y\rightarrow x}{\longrightarrow}(1-\chi_{B_\rho(x)}(z))\Big[|\psi(x)-\psi(z)|^{p-2}(\psi(x)-\psi(z))K_{sp}(x,z) \\
&\qquad\qquad\quad\qquad+a(x,z)|\psi(x)-\psi(z)|^{q-2}(\psi(x)-\psi(z))K_{tq}(x,z)\Big]
\end{align*}
for almost everywhere $z\in \mathbb{R}^n$, due to the continuity of $a(\cdot,z)$, $K_{sp}(\cdot,z)$ and $K_{tq}(\cdot,z)$. Afterwards, in view of $a(x,y)>0$, we evaluate
\begin{align*}
&\quad (1-\chi_{B_\rho(y)}(z))\Big[|\psi(y)-\psi(z)|^{p-2}(\psi(y)-\psi(z))K_{sp}(y,z)\\
&\qquad\qquad\qquad\quad+a(y,z)|\psi(y)-\psi(z)|^{q-2}(\psi(y)-\psi(z))K_{tq}(y,z)\Big]\\
&\leq(1-\chi_{B_{\rho/2}(x)}(z))\Big[2(|\psi(x)-\psi(z)|^{p-1}+1)K_{sp}(x,z)\\
&\qquad\qquad\qquad\quad+4a(x,z)(|\psi(x)-\psi(z)|^{q-1}+1)K_{tq}(x,z)\Big],
\end{align*}
when $y$ is sufficiently close to $x$. Hence it follows, from the dominated convergence theorem as well as the condition that $\psi\in L^{p-1}_{sp}(\mathbb{R}^n)\cap L^{q-1}_{a,tq}(\mathbb{R}^n)$, that
\begin{align}
\label{4-7}
&\quad\int_{\mathbb{R}^n\setminus B_\rho(y)}\Big[|\psi(y)-\psi(z)|^{p-2}(\psi(y)-\psi(z))K_{sp}(y,z) \nonumber\\
&\qquad\qquad\quad+a(y,z)|\psi(y)-\psi(z)|^{q-2}(\psi(y)-\psi(z))K_{tq}(y,z)\Big]\,dz \nonumber\\
&\rightarrow\int_{\mathbb{R}^n\setminus B_\rho(x)}\Big[|\psi(x)-\psi(z)|^{p-2}(\psi(x)-\psi(z))K_{sp}(x,z) \nonumber\\
&\qquad\qquad\quad+a(x,z)|\psi(x)-\psi(z)|^{q-2}(\psi(x)-\psi(z))K_{tq}(x,z)\Big]\,dz
\end{align}
by sending $y\rightarrow x$. Merging the display \eqref{4-5}--\eqref{4-7}, it yields that $|\mathcal{L}\psi(y)-\mathcal{L}\psi(x)|<\varepsilon$, whenever $|y-x|$ is small enough. We now finish the proof.
\end{proof}

The forthcoming lemma formulates the continuity of operator $\mathcal{L}$ regarding perturbations that are regular enough.

\begin{lemma}
\label{lem4-6}
Let $B_r(x_0)\subset\Omega$ and $\psi\in C^2(B_r(x_0))\cap L^{p-1}_{sp}(\mathbb{R}^n)\cap L^{q-1}_{a,tq}(\mathbb{R}^n)$. If $1<p\leq \frac{2}{2-s}$ and $D\psi(x_0)=0$, we suppose that $x_0$ is an isolated critical point and $\psi\in C^2_\beta(B_r(x_0))$ with $\beta>\frac{sp}{p-1}$. Then, for each $\varepsilon>0$ and $\varrho>0$, there are $\vartheta>0$, $\rho\in(0,\varrho)$ and $\eta\in C^2_0(B_{\rho/2}(x_0))$ with $0\leq \eta\leq1$ and $\eta(x_0)=1$ such that, when $0\leq\theta<\vartheta$, $\psi_\theta\equiv\psi+\theta\eta$ fulfills
$$
\sup_{x\in B_r(x_0)}|\mathcal{L}\psi(x)-\mathcal{L}\psi_\theta(x)|<\varepsilon.
$$
\end{lemma}

\begin{proof}
Let $\varepsilon,\varrho>0$. If $D\psi(x_0)\neq0$, then by continuity there are $\tau>0$ and $\rho\in(0,\varrho)$ satisfying $|D\psi|>\tau$ in $B_{2\rho}(x_0)$. Let $\eta\in C^2_0(B_{\rho/2}(x_0))$ be such that $0\leq \eta\leq1$ and $\eta(x_0)=1$. Thus it yields that $|D\psi_\theta|>\frac{\tau}{2}$ in $B_{2\rho}(x_0)$, provided that $0\leq \theta<\overline{\vartheta}$ for some $\overline{\vartheta}>0$. At this point, applying Lemma \ref{lem4-3}, we could choose $\delta>0$ so small that, for any $x\in B_{\rho}(x_0)$ and $\theta\in[0,\overline{\vartheta})$, there holds that
\begin{align}
\label{4-8}
\Bigg|\mathrm{P.V.}\int_{B_\delta(x)}\Big[&|\psi_\theta(x)-\psi_\theta(z)|^{p-2}(\psi_\theta(x)-\psi_\theta(z))K_{sp}(x,z) \nonumber\\
&+a(x,z)|\psi_\theta(x)-\psi_\theta(z)|^{q-2}(\psi_\theta(x)-\psi_\theta(z))K_{tq}(x,z)\Big]\,dz\Bigg|<\frac{\varepsilon}{3}.
\end{align}
Additionally, we can get the inequality \eqref{4-8} via Lemma \ref{lem4-3} as well, in the case $p>\frac{2}{2-s}$.

Next, we are going to check the scenario that $1<p\leq \frac{2}{2-s}$ and $D\psi(x_0)=0$. It is well known that $x_0$ is an isolated critical point of $\psi$, so we can see that $D\psi\neq0$ in $B_{2\rho}(x_0)\setminus\{x_0\}$ with $\rho>0$ small enough. We take a smooth function $\eta\in C^2_0(B_{\rho/2}(x_0))$ such that $0\leq\eta\leq1$ and $\eta\equiv1$ in $B_{\rho/4}(x_0)$ and $|D^2\eta|\leq Md_\eta^{\beta-2}$ with some constant $M>0$. Hence we can see that $D\psi_\theta\neq0$ in $B_{2\rho}(x_0)\setminus\{x_0\}$ if $\theta$ is sufficiently small, and further $d_\psi=d_{\psi_\theta}$ in $B_{\rho}(x_0)$ for every $\theta$ as above, as well as $\frac{1}{2}|D\psi|\leq|D\psi_\theta|\leq2|D\psi|$ in $B_{\rho}(x_0)$. Employing the fact that $d_\eta\leq d_\psi=d_{\psi_\theta}$ in $B_{\rho}(x_0)$, we infer easily that $|D^2\psi_\theta|\leq Cd_{\psi_\theta}^{\beta-2}$ in $B_{\rho}(x_0)$ if $\theta$ is small enough. So we know $\psi_\theta\in C^2_\beta(B_{\rho}(x_0))$ and further by means of Lemma \ref{lem4-4} we obtain \eqref{4-8} for $\delta\in(0,\rho)$ sufficiently small.

Now set $x\in B_\rho(x_0)$. It follows from \eqref{4-8} and Proposition \ref{pro2-2} that
\begin{align*}
&\quad |\mathcal{L}\psi(x)-\mathcal{L}\psi_\theta(x)|\\
&\leq \frac{2}{3}\varepsilon+\int_{\mathbb{R}^n\setminus B_\delta(x)}\Big[|g_p(\psi(x)-\psi(y))-g_p(\psi_\theta(x)-\psi_\theta(y))|K_{sp}(x,y)\\
&\qquad\qquad\qquad\quad+a(x,y)|g_q(\psi(x)-\psi(y))-g_q(\psi_\theta(x)-\psi_\theta(y))|K_{tq}(x,y)\Big]\,dy\\
&\leq \frac{2}{3}\varepsilon+C\theta\int_{\mathbb{R}^n\setminus B_\delta(x)}\left[\frac{(|\psi(x)-\psi(y)|+2\theta)^{p-2}}{|x-y|^{n+sp}}+a(x,y)
\frac{(|\psi(x)-\psi(y)|+2\theta)^{q-2}}{|x-y|^{n+tq}}\right]\,dy.
\end{align*}
When $1<p\leq q<2$, we proceed with estimating
\begin{align*}
&\quad |\mathcal{L}\psi(x)-\mathcal{L}\psi_\theta(x)|\\
&\leq \frac{2}{3}\varepsilon+C\int_{\mathbb{R}^n\setminus B_\delta(x)}\frac{\theta^{p-1}}{|x-y|^{n+sp}}+a(x,y)
\frac{\theta^{q-1}}{|x-y|^{n+tq}}\,dy\\
&\leq \frac{2}{3}\varepsilon+C(\theta^{p-1}\delta^{-sp}+M\theta^{q-1}\delta^{-tq})<\varepsilon,
\end{align*}
as long as $\theta$ is sufficiently small. On the other hand, for $2\leq p\leq q$, we have
\begin{align*}
&\quad |\mathcal{L}\psi(x)-\mathcal{L}\psi_\theta(x)|\\
&\leq \frac{2}{3}\varepsilon+C\int_{\mathbb{R}^n\setminus B_\delta(x)}\frac{\theta(\theta^{p-2}+|\psi(x)|^{p-2}+|\psi(y)|^{p-2})}{|x-y|^{n+sp}}\,dy\\
&\quad+C\int_{\mathbb{R}^n\setminus B_\delta(x)}a(x,y)\frac{\theta(\theta^{q-2}+|\psi(x)|^{q-2}+|\psi(y)|^{q-2})}{|x-y|^{n+tq}}\,dy\\
&\leq \frac{2}{3}\varepsilon+C\left(\theta^{p-1}\delta^{-sp}+\theta\delta^{-sp}\sup_{B_\rho(x_0)}|\psi|^{p-2}+\theta\delta^{-sp}\sup_{z\in B_\rho(x_0)}[\mathrm{Tail}(\psi;z,\delta)]^{p-2}\right)\\
&\quad+CM\left(\theta^{q-1}\delta^{-tq}+\theta\delta^{-tq}\sup_{B_\rho(x_0)}|\psi|^{q-2}\right)+CM^\frac{1}{q-1}\theta\delta^{-tq}\sup_{z\in B_\rho(x_0)}[\mathrm{Tail}_a(\psi;z,\delta)]^{q-2}<\varepsilon,
\end{align*}
as long as $\theta$ is sufficiently small. Here we observe that $\psi\in L^{p-1}_{sp}(\mathbb{R}^n)\cap L^{q-1}_{a,tq}(\mathbb{R}^n)$ implies that both $\sup_{z\in B_\rho(x_0)}[\mathrm{Tail}(\psi;z,\delta)]^{p-2}$ and $\sup_{z\in B_\rho(x_0)}[\mathrm{Tail}_a(\psi;z,\delta)]^{q-2}$ are finite. Finally, in the case $1<p<2\leq q$, we can readily get
\begin{align*}
&\quad |\mathcal{L}\psi(x)-\mathcal{L}\psi_\theta(x)|\\
&\leq \frac{2}{3}\varepsilon+C\left(\theta^{p-1}\delta^{-sp}+M\theta^{q-1}\delta^{-tq}+M\theta\delta^{-tq}\sup_{B_\rho(x_0)}|\psi|^{q-2}\right)\\
&\quad+CM^\frac{1}{q-1}\theta\delta^{-tq}\sup_{z\in B_\rho(x_0)}[\mathrm{Tail}_a(\psi;z,\delta)]^{q-2}<\varepsilon,
\end{align*}
if $\theta$ is sufficiently small.

In all cases, we arrive at
$$
|\mathcal{L}\psi(x)-\mathcal{L}\psi_\theta(x)|<\varepsilon
$$
for any $x\in B_\rho(x_0)$, whenever $\theta$ is small enough. By taking the supremum over $B_\rho(x_0)$, this assertion follows.
\end{proof}

The next result states that a $C^2$-regular supersolution is a weak supersolution.

\begin{lemma}
\label{lem4-7}
Let $u\in C^2(B_r(x_0))\cap L^{p-1}_{sp}(\mathbb{R}^n)\cap L^{q-1}_{a,tq}(\mathbb{R}^n)$. If $1<p\leq \frac{2}{2-s}$ and $Du(x_0)=0$, we suppose that $u\in C^2_\beta(B_r(x_0))$ with $\beta>\frac{sp}{p-1}$. Furthermore, assume that $\mathcal{L}u\geq0$ in the pointwise sense in $B_r(x_0)$. Then we infer that $u$ is a continuous weak supersolution in $B_r(x_0)$.
\end{lemma}

\begin{proof}
Clearly, $u\in W^{s,p}_{\rm loc}(B_r(x_0))$. Let $\psi\in C^\infty_0(B_r(x_0))$ be nonnegative. From $\mathcal{L}u\geq0$, we know that, for $x\in \mathrm{supp}\,\psi$ and $\varepsilon>0$,
\begin{align*}
\int_{\mathbb{R}^n\setminus B_\varepsilon(x)}&\Big[|u(x)-u(y)|^{p-2}(u(x)-u(y))K_{sp}(x,y)\\
&+a(x,y)|u(x)-u(y)|^{q-2}(u(x)-u(y))K_{tq}(x,y)\Big]\,dy\geq -\theta_\varepsilon(x),
\end{align*}
where $\theta_\varepsilon(x)$ tends to 0 uniformly as $\varepsilon\rightarrow0$ in view of the continuity of $\mathcal{L}u$ (Lemma \ref{lem4-5}). According to the previous inequality, it follows that
\begin{align*}
\int_{\mathbb{R}^n}\int_{\mathbb{R}^n}&(1-\chi_{B_\varepsilon(x)}(y))\Big[|u(x)-u(y)|^{p-2}(u(x)-u(y))K_{sp}(x,y)\\
&+a(x,y)|u(x)-u(y)|^{q-2}(u(x)-u(y))K_{tq}(x,y)\Big]\psi(x)\,dy\,dx\geq -\int_{\mathbb{R}^n}\theta_\varepsilon(x)\psi(x)\,dx.
\end{align*}
Exchanging the roles of $x$ and $y$, via the symmetry of functions $a,K_{sp},K_{tq}$ we get
\begin{align*}
\int_{\mathbb{R}^n}\int_{\mathbb{R}^n}&(1-\chi_{B_\varepsilon(y)}(x))\Big[|u(y)-u(x)|^{p-2}(u(y)-u(x))K_{sp}(x,y)\\
&+a(x,y)|u(y)-u(x)|^{q-2}(u(y)-u(x))K_{tq}(x,y)\Big]\psi(y)\,dx\,dy\geq -\int_{\mathbb{R}^n}\theta_\varepsilon(x)\psi(x)\,dx.
\end{align*}
Adding up the above two inequalities, it yields that
\begin{align}
\label{4-9}
&\int_{\mathbb{R}^n}\int_{\mathbb{R}^n\setminus B_\varepsilon(y)}\Big[|u(x)-u(y)|^{p-2}(u(x)-u(y))K_{sp}(x,y) \nonumber\\
&\quad+a(x,y)|u(x)-u(y)|^{q-2}(u(x)-u(y))K_{tq}(x,y)\Big](\psi(x)-\psi(y))\,dx\,dy\geq -2\|\theta_\varepsilon\psi\|_{L^1(B_r(x_0))}.
\end{align}
It is easy to know that $\|\theta_\varepsilon\psi\|_{L^1(B_r(x_0))}\rightarrow0$ as $\varepsilon\rightarrow0$. We next check that the integrand of the integration in the left-hand side is bounded by an integrable function. Let $\mathrm{supp}\,\psi\subset B_\rho\subset\subset B_r(x_0)$. Applying Young's inequality and the assumption ($A_2$), we can estimate
\begin{align*}
&\quad\int_{\mathbb{R}^n}\int_{\mathbb{R}^n}\Big[|u(x)-u(y)|^{p-1}K_{sp}(x,y)+a(x,y)|u(x)-u(y)|^{q-1}K_{tq}(x,y)\Big]|\psi(x)-\psi(y)|\,dx\,dy\\ &\leq C\int_{B_\rho}\int_{B_\rho}\left(\frac{|u(x)-u(y)|^{p-1}|\psi(x)-\psi(y)|}{|x-y|^{n+sp}}+a(x,y)\frac{|u(x)-u(y)|^{q-1}|\psi(x)-\psi(y)|}{|x-y|^{n+tq}}\,dx\right)\,dy\\
&\quad+C\int_{\mathbb{R}^n\setminus B_\rho}\int_{\mathrm{supp}\,\psi}\left(\frac{|u(x)-u(y)|^{p-1}\psi(x)}{|x-y|^{n+sp}}+a(x,y)\frac{|u(x)-u(y)|^{q-1}\psi(x)}{|x-y|^{n+tq}}\right)\,dx\,dy\\
&\leq C\int_{B_\rho}\int_{B_\rho}\left(\frac{|u(x)-u(y)|^p}{|x-y|^{n+sp}}+a(x,y)\frac{|u(x)-u(y)|^q}{|x-y|^{n+tq}}\right)\,dx\,dy\\
&\quad+C\int_{B_\rho}\int_{B_\rho}\left(\frac{|\psi(x)-\psi(y)|^p}{|x-y|^{n+sp}}+a(x,y)\frac{|\psi(x)-\psi(y)|^q}{|x-y|^{n+tq}}\right)\,dx\,dy\\
&\quad+C\int_{\mathbb{R}^n\setminus B_d(z)}\int_{\mathrm{supp}\,\psi}\left(\frac{|u(x)|^{p-1}\psi(x)}{|z-y|^{n+sp}}+a(x,y)\frac{|u(x)|^{q-1}\psi(x)}{|z-y|^{n+tq}}\right)\,dx\,dy\\
&\quad+C\int_{\mathbb{R}^n\setminus B_d(z)}\int_{\mathrm{supp}\,\psi}\left(\frac{|u(y)|^{p-1}\psi(x)}{|z-y|^{n+sp}}+a(x,y)\frac{|u(y)|^{q-1}\psi(x)}{|z-y|^{n+tq}}\right)\,dx\,dy\\
&=:I_1+I_2+I_3+I_4,
\end{align*}
where $z\in\mathrm{supp}\,\psi$ and $d:=\mathrm{dist}(z,\partial B_\rho)$. We now verify that these integrals $I_1,I_2,I_3,I_4$ are finite. First, by virtue of the regularity for $u,\psi$, it is easy to know that $I_1,I_2$ are finite quantities. Second, we deal with $I_3$ as follows,
\begin{equation*}
I_3\leq Cd^{-sp}\int_{\mathrm{supp}\,\psi}|u|^{p-1}\psi\,dx+CMd^{-tq}\int_{\mathrm{supp}\,\psi}|u|^{q-1}\psi\,dx<\infty.
\end{equation*}
Finally, we estimate
\begin{align*}
&\quad\int_{\mathbb{R}^n\setminus B_d(z)}\int_{\mathrm{supp}\,\psi}a(x,y)\frac{|u(y)|^{q-1}\psi(x)}{|z-y|^{n+tq}}\,dx\,dy\\
&\leq \|\psi\|_{L^\infty(B_r(x_0))}\int_{\mathbb{R}^n\setminus B_d(z)}\int_{\mathrm{supp}\,\psi}a(x,y)\,dx\frac{|u(y)|^{q-1}}{|z-y|^{n+tq}}\,dy\\
&=\|\psi\|_{L^\infty(B_r(x_0))}\int_{\mathbb{R}^n\setminus B_d(z)}a(\xi,y)\frac{|u(y)|^{q-1}}{|z-y|^{n+tq}}\,dy\\
&\leq \|\psi\|_{L^\infty(B_r(x_0))}d^{-tq}[\mathrm{Tail}_a(u;z,d)]^{q-1},
\end{align*}
where we have used the mean value theorem due to the continuity of $a(\cdot,y)$ and $\xi\in \mathrm{supp}\,\psi$. As a consequence, it follows that
\begin{equation*}
I_4\leq C\|\psi\|_{L^\infty(B_r(x_0))}d^{-sp}[\mathrm{Tail}(u;z,d)]^{p-1}+C\|\psi\|_{L^\infty(B_r(x_0))}d^{-tq}[\mathrm{Tail}_a(u;z,d)]^{q-1}<\infty.
\end{equation*}
Therefore, we can apply the dominated convergence to arrive at
\begin{align*}
\int_{\mathbb{R}^n}\int_{\mathbb{R}^n}\Big[&|u(x)-u(y)|^{p-2}(u(x)-u(y))K_{sp}(x,y)\\
&+a(x,y)|u(x)-u(y)|^{q-2}(u(x)-u(y))K_{tq}(x,y)\Big](\psi(x)-\psi(y))\,dx\,dy\geq 0,
\end{align*}
by sending $\varepsilon\rightarrow0$ in inequality \eqref{4-9}, which leads to $u$ being a weak supersolution.
\end{proof}

Finally, we conclude this section by proving that any bounded weak supersolution to \eqref{main} is viscosity supersolution. This proof is completed by a contradiction argument, where we shall make use of Lemma \ref{lem4-7}, the continuity of operator $\mathcal{L}$ (Lemmas \ref{lem4-5} and \ref{lem4-6}) together with the comparison principle for weak solutions (Proposition \ref{pro4-1}).

\begin{theorem}
\label{thm4-8}
Let $u$ be a bounded and lower semicontinuous weak supersolution to Eq. \eqref{main}. Then $u$ is a viscosity supersolution to \eqref{main}.
\end{theorem}

\begin{proof}
Assume that $u$ is a bounded weak supersolution with lower semicontinuity. According to the definition of viscosity solutions, the only property left to show is property (iii) in Definition \ref{def1}. In order to demonstrate that $u$ is a viscosity supersolution, we choose $\psi\in C^2(B_r(x_0))$ such that $\psi(x_0)=u(x_0)$, $\psi\leq u$ in $B_r(x_0)$ and that either (a) or (b) with $\beta>\frac{sp}{p-1}$ in Definition \ref{def1} (iii) holds. Then we have to show
\begin{equation}
\label{4-10}
\mathcal{L}\psi_r(x_0)\geq 0,
\end{equation}
where
$$
\psi_r=\begin{cases} \psi  &\text{\textmd{in } } B_r(x_0), \\[2mm]
u   &\text{\textmd{in } } \mathbb{R}^n\setminus B_r(x_0).
\end{cases}
$$
We argue by contradiction. If \eqref{4-10} is not true, then, by the continuity of $\mathcal{L}\psi_r$ (see Lemma \ref{lem4-5}), for some $\tau>0$ and $\varrho\in (0,r)$ we obtain
$$
\mathcal{L}\psi_r\leq -\tau  \quad\text{in } B_\varrho(x_0).
$$
In addition, it follows from Lemma \ref{lem4-6} that there are $\theta>0$, $\rho\in(0,\varrho)$ and $\eta\in C^2_0(B_{\rho/2}(x_0))$ with $0\leq\eta\leq1$, $\eta(x_0)=1$, such that $\varphi:=\psi_r+\theta\eta$ fulfills
$$
\sup_{B_{\rho}(x_0)}|\mathcal{L}\psi_r-\mathcal{L}\varphi|<\tau.
$$
Therefore, we get
$$
\mathcal{L}\varphi\leq0 \quad\text{in } B_\rho(x_0).
$$
From Lemma \ref{lem4-7}, we know that $\varphi$ is a continuous weak subsolution in $B_\rho(x_0)$. Obviously, $\varphi=\psi_r\leq u$ in $\mathbb{R}^n\setminus B_{\rho/2}(x_0)$. Hence through Comparison principle \ref{pro4-1}, it yields that
$$
\varphi\leq u \quad\text{in } B_{\rho/2}(x_0).
$$
In particular, $\varphi(x_0)=\psi(x_0)+\theta\leq u(x_0)$ ($\theta>0$), which contradicts $\psi(x_0)=u(x_0)$. We now have showed that $u$ is a viscosity supersolution.
\end{proof}

\section*{Acknowledgments}

This work was supported by the National Natural Science Foundation of China (No. 12071098).

\end{document}